\documentclass{aptpub}

\usepackage{color}
\usepackage[dvipsnames]{xcolor}
\usepackage{pgf,tikz}
\usepackage{graphicx}
\usepackage{dsfont}
\usepackage{enumerate}
\usepackage[normalem]{ulem}

\newcommand{\mylabel}[2]{#2\def\@currentlabel{#2}\label{#1}}

\authornames{J.~KREBS AND W.~POLONIK} 
\shorttitle{Asymptotic normality of persistent Betti numbers} 

\usepackage{mathtools}

\DeclarePairedDelimiter\ceil{ \lceil}{ \rceil}

\newcommand{\N}{\mathbb{N}}
\newcommand{\R}{\mathbb{R}}

\newcommand{\mX}{\mathbb{X}}
\newcommand{\mY}{\mathbb{Y}}
\newcommand{\Z}{\mathbb{Z}}

\newcommand{\p}{\mathbb{P}}

\newcommand{\cA}{\mathcal{A}}
\newcommand{\cB}{\mathcal{B}}
\newcommand{\cC}{\mathcal{C}}

\newcommand{\cF}{\mathcal{F}}
\newcommand{\cH}{\mathcal{H}}

\newcommand{\cJ}{\mathcal{J}}
\newcommand{\cK}{\mathcal{K}}
\newcommand{\cL}{\mathcal{L}}

\newcommand{\cN}{\mathcal{N}}

\newcommand{\cP}{\mathcal{P}}
\newcommand{\cR}{\mathcal{R}}

\newcommand{\cU}{\mathcal{U}}
\newcommand{\cV}{\mathcal{V}}
\newcommand{\cW}{\mathcal{W}}

\newcommand{\cX}{\mathcal{X}}

\newcommand{\fC}{\mathfrak{C}}
\newcommand{\fD}{\mathfrak{D}}

\newcommand{\fK}{\mathfrak{K}}

\newcommand{\fR}{\mathfrak{R}}
\newcommand{\fS}{\mathfrak{S}}

\newcommand{\fQ}{\mathfrak{Q}}
\newcommand{\fV}{\mathfrak{V}}
\newcommand{\E}[1]{\mathbb{E}\left [  #1 \right ]}

\newcommand{\V}[1]{\operatorname{Var}\left [  #1  \right ] }

\renewcommand{\epsilon}{\varepsilon}
\renewcommand{\phi}{\varphi}

\newcommand{\intd}[1]{\mathrm{d}#1}

\newcommand{\1}[1]{\,\mathds{1}\! \left\{ #1 \right\} }
\newcommand{\covar}[2]{\operatorname{Cov} \left(#1,#2		\right) }

\newcommand{\poi}{\operatorname{Poi}}
\newcommand{\diam}{\operatorname{diam}}

\newcommand{\ul}{\underline}
\newcommand{\ol}{\overline}
\newcommand{\wt}{\widetilde}

\newcommand{\cPc}{\cP^\circ}
\newcommand{\cPd}{\cP^\dagger}
\newcommand{\omegad}{\omega^\dagger}
\newcommand{\omegac}{\omega^\circ}
\newcommand{\diff}{\mathrm{d}}

\allowdisplaybreaks

\begin{document}

\title{On the asymptotic normality of persistent Betti numbers} 

\authorone[KU Eichstätt-Ingolstadt]{Johannes Krebs}
\authortwo[University of California at Davis]{Wolfgang Polonik}

\addressone{KU Eichstätt-Ingolstadt, Ostenstra\ss e 28, 85072 Eichst{\"a}tt, Germany} 
\emailone{johannes.krebs@ku.de} 

\addresstwo{Department of Statistics, University of California, Davis, CA, 95616, USA} 
\emailtwo{wpolonik@ucdavis.edu} 

\vspace{-3em}

\begin{abstract}
Persistent Betti numbers are a major tool in persistent homology, a subfield of topological data analysis. Many tools in persistent homology rely on the properties of persistent Betti numbers considered as a two-dimensional stochastic process $ (r,s) \mapsto n^{-1/2} (\beta^{r,s}_q ( \cK(n^{1/d} \cX_n))-\mathbb{E}[\beta^{r,s}_q ( \cK( n^{1/d} \cX_n))])$. So far, pointwise limit theorems have been established in different settings. In particular, the pointwise asymptotic normality of (persistent) Betti numbers has been established for stationary Poisson processes and binomial processes with constant intensity function in the so-called critical (or thermodynamic) regime, see Yogeshwaran et al.\ \cite{yogeshwaran2017random} and Hiraoka et al.\ \cite{hiraoka2018limit}. 

In this contribution, we derive a strong stabilization property (in the spirit of Penrose and Yukich \cite{penrose2001central}) of persistent Betti numbers and generalize the existing results on the asymptotic normality to the multivariate case and to a broader class of underlying Poisson and binomial processes. Most importantly, we show that the multivariate asymptotic normality holds for all pairs $(r,s)$, $0\le r\le s<\infty$, and that it is not affected by percolation effects in the underlying random geometric graph.
\end{abstract}



\section{Introduction}
In this manuscript we address an important question in topological data analysis (TDA), namely, the study of the weak convergence of persistent Betti numbers
\begin{align}\label{E:GaussianLimit}
		\left( 	 n^{-1/2}\left( \beta^{r_i,s_i}_q(\cK(n^{1/d} \cX_n) ) - \E{ \beta^{r_i,s_i}_q(\cK(n^{1/d} \cX_n) ) } \right): 1\le i \le \ell \right),
\end{align}
where $0\le q \le d-1$ and $0\le r_i\le s_i <\infty$ for $1\le i \le \ell$ ($\ell\in\N$) and where $\cX_n$ is either an  $n$-binomial process with a bounded density $\kappa$ defined on the unit cube $[0,1]^d$ or the corresponding Poisson process with intensity function $n \kappa$ for $n\in\N$.

So far, there exist results on the pointwise asymptotic normality for Betti numbers (i.e., $\ell = 1$) in the case of a homogeneous Poisson process or a binomial process with a constant density, see Yogeshwaran et al.\ \cite{yogeshwaran2017random}. In the case of a homogenous Poisson process this result was extended to persistent Betti numbers by Hiraoka et al.\ \cite{hiraoka2018limit}.

Based on the pioneering central limit theorem of Penrose and Yukich \cite{penrose2001central} for stabilizing functionals on the homogeneous Poisson process, Trinh \cite{trinh2019central} extends the central limit theorem to strongly stabilizing functionals in case of an underlying inhomogeneous Poisson process. We will apply the abstract result of Trinh to persistent Betti functions (see below). For this we establish the strong stabilization property of the persistent Betti function, and this is one of our main contributions.

The theory on random geometric complexes is growing rapidly. For pioneering contributions see Kahle \cite{kahle2011random}, Yogeshwaran and Adler \cite{yogeshwaran2015topology} or Owada and Adler \cite{owada2017limit}. We refer to Bobrowski and Kahle \cite{bobrowski2014topology} for a survey.

Recent contributions in the context of Betti numbers are Kahle and Meckes \cite{kahle2010limit}, Owada \cite{owada2018limit}, Goel et al.\ \cite{goel2018asymptotic}, Divol and Polonik \cite{divol2018persistence} and Owada and Thomas \cite{owada2019limit}.

TDA is a comparably young field that has emerged from several contributions in algebraic topology and computational geometry. Milestone contributions which helped to popularize TDA in its early days are Edelsbrunner et al.\ \cite{edelsbrunner2000topological}, Zomorodian and Carlsson \cite{zomorodian2005computing} as well as Carlsson \cite{carlsson2009topology}. Given a point cloud sampled from a distribution with density $f$ on a $d$-dimensional manifold, TDA consists of various techniques that aim at understanding the topology of the manifold and of the density $f$. The various methods of TDA have been successfully implemented in applied sciences such as biology (\cite{yao2009topological}), materials sciences (\cite{lee2017quantifying}) or chemistry (\cite{nakamura2015persistent}). From the mathematical statistician's point of view, a particular interest deserves the application of TDA to time series, see, e.g., the pioneering works of Seversky et al.\ \cite{seversky2016time}, Umeda \cite{umeda2017time} and the contributions of Gidea et al.\ to the analysis of financial time series (\cite{gidea2018topological, gidea2017topological, gidea2019topological}).

Our contribution falls into the area of persistent homology which is one of the major tools in TDA. We can only give a short introduction to this topic here, a more detailed introduction offering insights to the basic concepts, ideas and applications of persistent homology can be found in Chazal and Michel \cite{chazal2021introduction}, Oudot \cite{oudot2015persistence} and Wasserman \cite{wasserman2018topological}.

Given a point cloud in $\R^d$ (a sample of a point process) one first builds simplicial complexes over this point cloud according to a rule that describes the neighborhood relation between points. The two most frequent simplicial complex models are the Vietoris-Rips and the {\v C}ech complex defined below. When considered as geometric structures, topological properties of simplicial complexes are characterized by the number of their $q$-dimensional holes, most notably connected components, loops and cavities (0, 1 and 2 dimensional features). These holes are theoretically defined with a tool from algebraic topology, the so-called homology. The $q$th homology of a simplicial complex is determined by a quotient space. Its dimension is the so-called $q$th Betti number. Intuitively, the $q$th Betti number counts the number of $q$-dimensional holes in the simplicial complex.

For a given simplicial complex model, a filtration is an increasing collection of simplicial complexes indexed by a one-dimensional parameter, the so-called filtration parameter, which can be understood as time. Given a filtration on a finite time interval, we can consider the evolution of the $q$th homology groups, i.e., the dynamical behavior of the Betti numbers. As the underlying simple point process (e.g., a Poisson process on a Euclidean space) is random, these Betti numbers are also random and we consider a stochastic process. 

From the applied point of view, the mere knowledge of the evolution of the Betti numbers is often not enough, especially when considering objects obtained from persistence diagrams, such as persistent landscapes. In this context the more general concept of {\em persistent} Betti numbers is the appropriate tool, and this is the object studied here. 

The remainder of this manuscript is organized as follows. In Section~\ref{Section_DefinitionsNotation} we introduce our notation. The main results are presented in Section~\ref{Section_MainResults}. We state the property of strong stabilization and present two central limit theorems for persistent Betti numbers. Section~\ref{Section_RelatedResults} offers a short review of important related results in the literature that are also used in our studies. The framework of stabilization is treated in detail in Section~\ref{Section_Stabilization} that also contains further results on the stabilizing properties of persistent Betti numbers.
The technical details are given in Section~\ref{Section_TechnicalResults} and in Appendix~\ref{AppendixTightnessOfStabilization}.

\section{Notation}\label{Section_DefinitionsNotation}
Given a finite subset $P$ of the Euclidean space $\R^d,$ the {\v C}ech filtration $\cC(P)=(\cC_r(P):r\ge 0)$ and the Vietoris-Rips filtration $\cR(P)=(\cR_r(P):r\ge 0)$ are defined by
 \begin{align*}
				\cC_r(P) &= \big\{  \sigma \subseteq P, \bigcap_{x\in \sigma} B(x,r)\neq \emptyset \big\}, \\
				\cR_r(P) &= \{ \sigma \subseteq P, \diam(\sigma)\le r \},
\end{align*}
respectively, where $B(x,r) = \{y\in\R^d: \|x-y\|\le r\}$, $\|\cdot\|$ is the Euclidean norm and $\diam$ is the diameter of a subset of $\R^d$. Each $\sigma \subseteq P$ of size $q+1$ is a q-simplex, and for each $r\ge 0$, the collections of simplices $\cC_r(P)$ and $\cR_r(P)$, respectively, form simplicial complexes. Throughout this article, $\cK_r(P)$ denotes either the {\v C}ech or the Vietoris-Rips complex built from $P$ for some $r\ge 0$. The {\v C}ech or the Vietoris-Rips filtration $\cK(P)$ is the nested sequence of complexes $\cK_r(P)$ as $r$ goes from 0 to $+\infty$.

Consider a filtration $\cK(P)$ and a time $r\ge 0$. The chain group, generated by the $q$-dimensional simplices at time $r$ is $C_q( \cK_r(P))$. Write $Z_q(\cK_r(P))$ for the $q$th cycle group of the simplicial complex $\cK_r(P)$ and $B_q(\cK_r(P))$ for the $q$th boundary group, and let $H_q(\cK_r(P))$ be the homology of the simplicial complex $\cK_r(P)$ w.r.t.\ to the base field $\mathbb{F}_2=\{0,1\}$. A $q$-simplex $\sigma$ is positive in the filtration $\cK(P)$ if, upon its filtration time $r(\sigma)$ (the time it enters the complex), its inclusion to the simplicial complex $\cK_{r(\sigma)-}$ creates a $q$-dimensional cycle. Here $\cK_{r(\sigma)-}$ is the simplicial complex $\cK_{r(\sigma)}$ without the complexes containing $\sigma$ (as a simplex or as a face). If the $q$-simplex $\sigma$ is not positive, then we call it negative.

Let $0\le q \le d-1$. Then the $(r,s)$-persistent Betti number of a simplicial complex $\cK(P)$ (see \cite{edelsbrunner2000topological}) is defined by
\begin{align*}
			\beta^{r,s}_q(\cK(P)) &= \dim \frac{Z_q (\cK_r(P))}{Z_q (\cK_r(P)) \cap B_q (\cK_s(P))} \\
			&= \dim Z_q (\cK_r(P)) - \dim Z_q (\cK_r(P)) \cap B_q (\cK_s(P)).
\end{align*}
The Betti number $\beta^r_q(\cK(P))$ is defined as $\beta^{r,r}_q(\cK(P))$, $r\ge 0$; in particular, $\beta^r_q(\cK(P)) = \dim Z_q (\cK_r(P)) - \dim B_q (\cK_r(P))$. 

The persistent Betti number $\beta^{r,s}_q(\cK(P))$ is closely related to the persistence diagram of the underlying point cloud $P$, see \cite{hiraoka2018limit} for further details.  It equals the number of $q$-dimensional topological features (points in the $q$th persistence diagram) born before time $r$ and still alive at time $s$ (see Figure~\ref{F:Persistence diagram}). Persistent Betti numbers are translation invariant, i.e., $\beta^{r,s}_q(\cK(P+v)) = \beta^{r,s}_q(\cK(P))$ for each $v\in\R^d$. The add one cost function 
$$
	\fD_0 \beta^{r,s}_q(\cK(P)) = \beta^{r,s}_q(\cK(P \cup \{0\} )) - \beta^{r,s}_q(\cK(P))
	$$ is an important tool in our analysis. 

\begin{figure}[ht]
\centering
\begin{tikzpicture}[scale=.9] 
	\fill[fill=lightgray] (2,3)--(2,4.5)--(0,4.5)--(0,3);                 
	\draw[->] (0,0)--(0,5) node[above] {\scriptsize{death}};
	\draw[->] (0,0)--(5,0) node[right] {\scriptsize{birth}};
	\draw (0,0)--(5,5);
	\draw [dotted] (4.5,0)--(4.5,4.5);
	\draw [dotted] (0,4.5)--(4.5,4.5);
	\draw [red, thick] (2,3)--(2,4.5);
	\draw [dashed, red, thick] (2,3) -- (0,3);
	\draw [dotted] (2,0)--(2,3);
	\draw (-.3,3) node {s};
	\draw (2,-.3) node {r};
	\draw (-.3,4.5) node {$\infty$};
	\filldraw (1,2) circle (1.5pt);
	\filldraw (1.3,4) circle (1.5pt);
	\filldraw (1,4.2) circle (1.5pt);
	\filldraw (.5,3) circle (1.5pt);
		\filldraw (.4,2.8) circle (1.5pt);
	\filldraw (2.3,3) circle (1.5pt);
	\filldraw (4,4.4) circle (1.5pt);
	\filldraw (3.5,3.8) circle (1.5pt);
		\filldraw (2,3.3) circle (1.5pt);
	\filldraw (1,1.8) circle (1.5pt);
\end{tikzpicture}
\caption{$\beta^{r,s}_q(\cK(P))$ equals the number of points in the gray-shaded rectangle; the point on the dashed red line is not counted whereas the point on the solid red line is.}
\label{F:Persistence diagram}
\end{figure}

We let $\cP$ and $\cP'$ be two independent and homogeneous Poisson processes on $\R^d$ with unit intensity, observed on increasing observation windows $B_n=[-2^{-1} n^{1/d},2^{-1} n^{1/d}]^d$. Given a function $w \ge 0$, we denote by $\cP(w)$ the (inhomogenous) Poisson process with intensity function $w$.

We also use the following notation: $\Delta=\{(r,s): 0\le r\le s<\infty\}$ denotes the domain of the persistent Betti function. $Q(x,r) = \{y\in\R^d: |y_i - x_i| \le r \text{ for } 1\le i \le d \}$ and $Q(x) = (-1/2,1/2]^d + x$ for $x\in\R^d$ and $r>0$. For $y,z\in\Z^d$, we write  $y \prec z$ if $y$ precedes $z$ in the lexicographic ordering on $\Z^d$ and write $y \preceq z$ if either $y\prec z$ or $y=z$.
If $f\colon\R\rightarrow\R$, write $\|f\|_{\infty}$ for the sup-norm of $f$. We let $\Rightarrow$ denote convergence in distribution of a sequence of random variables. Throughout the article, let $(\Omega,\cF,\p)$ be a sufficiently rich probability space, on which all random variables are defined.

\section{Main results}\label{Section_MainResults}

We present the first main result, discussed in detail later in Section~\ref{Section_Stabilization}.  
\begin{theorem}\label{T:StrongStabAddOne}
Let $\lambda>0$, $(r,s)\in\Delta$ and $q \in \{0,\ldots,d-1\}$. There is an $\cF$-measurable random variable $S^{(r,s)}_q \coloneqq S^{(r,s)}_q(\cP(\lambda)) $ which is  $a.s.$ finite such that for all finite sets $A \subseteq \R^d\setminus B(0, S^{(r,s)}_q)$, the add one cost function satisfies
\begin{align*}
		&\fD_0 \beta^{r,s}_q\Big( \cK\big( \big(\,\cP(\lambda)\cap B(0,S^{(r,s)}_q)\,\big) \cup A \big) \Big) 
		\equiv \fD_0 \beta^{r,s}_q\Big( \cK\big( \cP(\lambda)\cap B(0,S^{(r,s)}_q)\big)\Big)  \quad a.s.
\end{align*}
\end{theorem}
Thus, the persistent Betti function is strongly stabilizing on the homogeneous Poisson process $\cP(\lambda)$ in the spirit of Penrose and Yukich \cite{penrose2001central} for each intensity $\lambda\in\R_+$, each pair $(r,s)\in\Delta$ and each dimension $q$. The proof of Theorem~\ref{T:StrongStabAddOne} relies on an abstract stabilization result stated in Theorem~\ref{Thrm:StrongStabilization}; see Section~\ref{Section_Stabilization}. The proofs of both theorems are given in Subsection~\ref{Subsection_Stabilization}.

By the property of strong stabilization there are random variables $\Delta^{r,s}_0(\infty)$ taking values in $\Z$ and $N_0$ taking values in $\N$ such that, for all $n \ge N_0$,
\begin{align*}
			&\beta_q^{r,s} (\cK(\cP\cap B_n)) - \beta_q^{r,s}  (\cK(( [\cP\setminus Q(0)]\cup[\cP'\cap Q(0)])\cap B_n) ) 
			\equiv \Delta^{r,s}_0(\infty),
\end{align*}
see Lemma~\ref{L:asStabilizationPoisson}. Let $\cF_0$ be the $\sigma$-field generated by $\cP$ restricted to $\bigcup_{y\in\Z^d: y  \preceq 0} Q(y)$.
Define $\gamma( (u,v), (r,s)) = \E{ \E{\Delta^{u,v}_0(\infty)|\cF_0} \E{\Delta^{r,s}_0(\infty)|\cF_0} }$.
The asymptotic normality in the Poisson sampling scheme can directly be derived from the the strong stabilization stated in Theorem~\ref{T:StrongStabAddOne} and the abstract result of Trinh \cite{trinh2019central} via the Cram\'{e}r -Wold device:

\begin{theorem}\label{T:GaussianLimitPoisson}
Let $\cP_n=\cP(n\kappa)$ be a Poisson process with intensity $n\kappa$ on $[0,1]^d$, where $\kappa$ is a bounded and measurable density function. Let $X \sim \kappa$ and let $(r_i,s_i)\in\Delta$ for $1\le i\le \ell$ and $\ell\in\N$. Then 
\[
			\begin{pmatrix}
				&n^{-1/2}\left( \beta^{r_1,s_1}_q(\cK(n^{1/d} \cP_n) ) - \E{ \beta^{r_1,s_1}_q(\cK(n^{1/d} \cP_n) ) } \right) \\
				&\vdots\\
				&n^{-1/2}\left( \beta^{r_\ell,s_\ell}_q(\cK(n^{1/d} \cP_n) ) - \E{ \beta^{r_\ell,s_\ell}_q(\cK(n^{1/d} \cP_n) ) } \right) 
				\end{pmatrix}
				 \Rightarrow \Psi,
\]
where $\Psi\sim \cN(0,\Sigma)$ has a multivariate normal distribution with mean zero and covariance matrix $\Sigma\, \ge 0$ given by
\[
		\Sigma(i,j) = \E{ \gamma( \kappa(X)^{1/d} ( (r_i,s_i),(r_j,s_j)) ) } \quad (1\le i,j\le \ell).
\]
Furthermore,
$$
	\lim_{n\to\infty} n^{-1} \covar{ \beta^{r_i,s_i}_q(\cK(n^{1/d} \cP_n) ) }{ \beta^{r_j,s_j}_q(\cK(n^{1/d} \cP_n) ) } = \Sigma(i,j) \qquad (1\le i,j\le \ell).
$$
\end{theorem}
\noindent Moreover, define for $0\le r\le s< \infty$
\begin{align}\begin{split}\label{Def:AlphaRS}
			\alpha(r,s) \coloneqq \E{	\fD_0 \beta^{r,s}_q\Big(\cK\big(	\cP \cap B(0,S^{(r,s)}_q ) \big) \Big)		}, 
\end{split}\end{align}
where $S^{(r,s)}_q = S^{(r,s)}_q(\cP)$ is as in Theorem~\ref{T:StrongStabAddOne}. Then, for the binomial sampling scheme the result is as follows:

\begin{theorem}\label{T:GaussianLimitBinomial}
Let $\mX_n$ be an $n$-binomial process with density function $\kappa$ on $[0,1]^d$, which is bounded and measurable. Let $X \sim \kappa$ and let $(r_i,s_i)\in\Delta$ for $1\le i \le \ell$ and $\ell\in\N$.  Then 
\[
			\begin{pmatrix}
			&n^{-1/2}\left( \beta^{r_1,s_1}_q(\cK(n^{1/d} \mX_n) ) - \E{ \beta^{r_1,s_1}_q(\cK(n^{1/d} \mX_n) ) } \right) \\
			&\vdots\\
			&n^{-1/2}\left( \beta^{r_\ell,s_\ell}_q(\cK(n^{1/d} \mX_n) ) - \E{ \beta^{r_\ell,s_\ell}_q(\cK(n^{1/d} \mX_n) ) } \right) 
			\end{pmatrix}
			\Rightarrow \wt\Psi,
\]
where $\wt\Psi\sim \cN(0,\wt\Sigma)$ has a multivariate normal distribution with mean zero and covariance matrix $\wt\Sigma\,\ge 0$ given by
\begin{align*}
		\wt\Sigma(i,j) &= \E{ \gamma( \kappa(X)^{1/d} ( (r_i,s_i),(r_j,s_j)) ) } \\
		&\qquad - \E{ \alpha( \kappa(X)^{1/d} (r_i,s_i))} \, \E{\alpha( \kappa(X)^{1/d} (r_j,s_j))} \quad (1\le i,j\le \ell).
\end{align*}
Furthermore,
$$
	\lim_{n\to\infty} n^{-1} \covar{ \beta^{r_i,s_i}_q(\cK(n^{1/d} \mX_n) ) }{ \beta^{r_j,s_j}_q(\cK(n^{1/d} \mX_n) ) } = \wt\Sigma(i,j) \qquad (1\le i,j\le \ell).
$$
\end{theorem}
We conclude this section with some discussion about the results and the techniques used in the corresponding proofs that are given in Subsections~\ref{S:Poisson} and \ref{S:Binomial}. The univariate central limit theorems for Betti numbers ($r=s$) have already been formulated by Trinh \cite{trinh2019central} (see also Proposition~\ref{PropositionHiraokaYogeshwaran} in this manuscript) under the condition that the parameter $r$ is chosen such that no percolation occurs. Now, as we can rely on the strong stabilization property, we can omit this restriction.

For the derivation of the multivariate results in the Poisson sampling scheme, we can rely on the abstract result of Trinh \cite[Theorem 3.3]{trinh2019central} and derive the covariance structure with the help of results from Penrose and Yukich \cite{penrose2001central}. Multivariate central limit theorems in the spatial context are also studied by Penrose \cite{penrose2005multivariate}. The results in the binomial setting are established with Trinh \cite[Theorem 3.9]{trinh2019central}, which itself relies on a de-Poissonization argument.

Finally, we mention that it is currently unknown whether or not the limiting covariance matrices are strictly positive definite.

\section{Related results}\label{Section_RelatedResults}
Below we discuss some literature closely related to our study. The techniques employed to obtain these results are tools from geometric probability, which studies geometric quantities deduced from simple point processes. A classical result of Steele \cite{steele1988growth} shows the convergence of the total length of the minimum spanning tree built from an i.i.d.\ sample of $n$ points in the unit cube. There are several generalizations of this work - for notable contributions see McGivney and Yukich \cite{mcgivney1999asymptotics}, Yukich \cite{yukich2000asymptotics}, Penrose and Yukich \cite{penrose2003weak} and the monograph of Penrose \cite{penrose2003random}. 

A different type of contribution, equally important, is Penrose and Yukich \cite{penrose2001central} which considers asymptotic normality of functionals built on Poisson and binomial processes. For completeness, we mention that the study of Gaussian limits (as in \cite{penrose2001central} and \cite{penrose2003weak}) is not limited to the total mass functional. It can be extended to random point measures obtained from the points of a marked point process, see, e.g., \cite{baryshnikov2005gaussian}, \cite{penrose2007gaussian} and \cite{blaszczyszyn2019limit}.

Goel et al.\ \cite{goel2018asymptotic} prove a convergence result for the expectation of Betti numbers in the critical regime. Their result generalizes directly to persistent Betti numbers and we have the following well-known result.
\begin{proposition}\label{P:ConvergenceExpectation}
Let $0<r\le s < \infty$. Let $\cX_n$ be either a Poisson process with intensity $n\kappa$ on $[0,1]^d$ or an $n$-binomial process  with density $ \kappa$. Then
\begin{align*}
			\lim_{n\rightarrow \infty} n^{-1} \E{ \beta^{r,s}_q( \cK( n^{1/d} \cX_n) ) } = \E{ \hat{b}_q(r \kappa(X')^{1/d},s \kappa(X')^{1/d})},
\end{align*}
where $X'$ has density $\kappa$ and where $\hat{b}_q(r,s)$ is the limit of $n^{-1} \E{ \beta^{r,s}_q(\cK((n^{1/d} \mX^*_n)) }$ for a homogeneous Poisson process $\mX^*_n$ on $[0,1]^d$ with intensity $n$.
\end{proposition}

So far, normality results for (persistent) Betti numbers exist only in a pointwise sense and are rather direct consequences of Theorem 2.1 and 3.1 given in \cite{penrose2001central}. We quote them here in a way which makes them more in line with our framework. For this we need the notion of the interval of co-existence $I_d(\cP)$ of a Poisson process $\cP$ with unit intensity on $\R^d$. They are defined by the critical radii for percolation of the occupied and the vacant component, respectively, which are defined as:
\[
	r_c(\cP) := \inf\{ r: \p(\cC_r(\cP) \text{ percolates})>0 \} 
\]
and
\[
	r^*_c(\cP) := \sup\{ r: \p( \R^d\setminus \cC_r(\cP) \text{ percolates})>0 \} .
\] 
Both probabilities inside the infimum and supremum are either 0 or 1 by Kolmogorov's 0-1-law. $r_c(\cP)$ is called the \textit{critical radius for percolation of the occupied component} and $r_c^*$ is called the \textit{critical radius for percolation of the vacant component}. The \textit{interval of co-existence} for which unbounded components of both the (Boolean) model $\cC_r(\cP)$ and its complement co-exist, is defined as follows:
\[
	I_d(\cP) := \begin{cases}
	(r_c, r^*_c] & \text{ if } \p( \cC_{r_c}(\cP) \text{ percolates}) = 0, \\
	[r_c,r^*_c] & otherwise.
	\end{cases}
\]
We know that in two dimensions, $I_2(\cP) = \emptyset$ from \cite[Theorem 4.4 and 4.5]{meester1996continuum}. Moreover, from \cite[Theorem 1]{sarkar1997co}, we know that $I_d(\cP)\neq\emptyset$ for each $d\ge 3$. Then, we have the following results.

\begin{proposition}[Pointwise normality of (persistent) Betti numbers]\label{PropositionHiraokaYogeshwaran}
\begin{itemize}
		\item [\mylabel{P:Hiraoka}{(i)}] Hiraoka et al.\ \cite[Theorem 5.2]{hiraoka2018limit}. Let $\cP|_{[0,n^{1/d}]^d}$ be the restriction of $\cP$ to $[0,n^{1/d}]^d$, and let $0\le r \le s<\infty$. Then there is a $\sigma^2(r,s)\in\R_+$ such that
		\[
					n^{-1/2} (\beta^{r,s}_q (\cK(\cP|_{[0,n^{1/d}]^d})) - \E{\beta^{r,s}_q (\cK(\cP|_{[0,n^{1/d}]^d})) } ) \Rightarrow \Phi_1,
		\]
		where $\Phi_1$ has a normal distribution with mean zero and variance $\sigma^2(r,s)\, \ge 0$.
		\item [(ii)] Yogeshwaran et al.\ \cite[Theorem 4.7]{yogeshwaran2017random}. Let $\cK$ be the {\v C}ech filtration and let $0\le r < \infty$ be such that $r \notin I_d(\cP)$. For each $n\in\N$, let $\mX_n$ be an $n$-binomial process with a uniform density on $[0,1]^d$. Then there is a $0<\tau^2(r) \le \sigma^2(r,r)$ (with $\sigma^2$ from (i)), such that
				\[
					n^{-1/2} \left(\beta^{r}_q (\cK( n^{1/d} \mX_n)) - \E{\beta^{r}_q (\cK(n^{1/d} \mX_n)) } \right) \Rightarrow \Phi_2,
		\]
		where $\Phi_2$ has a normal distribution with mean zero and variance $\tau^2(r)\,> 0$.
		\item [(iii)] Trinh \cite[Theorem 4.1]{trinh2019central}. Let $\kappa$ be a bounded density function with compact support, let $\cK$ be the {\v C}ech filtration. Let $0\le q\le d-1$. Let $r\in (0, (\sup \kappa)^{-1/d} \ r_c)$. Then
		$$
			n^{-1/2}\left( \beta^{r}_q(\cK(n^{1/d} \cP(n\kappa)) ) - \E{ \beta^{r}_q(\cK(n^{1/d} \cP(n\kappa) )) } \right) \Rightarrow \Phi_3, 
		$$
		where $\Phi_3$ has a normal distribution with mean zero and variance $\wt\sigma^2 > 0$, where \\ $\wt\sigma^2 = \int \sigma^2(\kappa(x)^{1/d}(r,r)) \kappa(x) \intd{x}$ with $\sigma^2$ from (i). A similar statement is true for the binomial process.
\end{itemize}
\end{proposition}
First, we remark that the above statements in their original versions are also valid for more general domains $\tilde B_n\subseteq\R^d$ which are not necessarily rectangular domains.
Furthermore, we remark that Hiraoka et al.\ \cite{hiraoka2018limit} prove their theorem for a general class of filtrations which contains, among others, the {\v C}ech and the Vietoris-Rips filtration. Moreover, Theorem 4.7 of Yogeshwaran et al.\ \cite{yogeshwaran2017random} also contains a version of (ii) for Betti numbers of the homogeneous Poisson process, which, in the above list, is contained in result (i).  The result of Trinh \cite{trinh2019central} is already for Betti numbers from a general density function $\kappa$ but the parameter choice for $r$ depends on $r_c$. Moreover, Trinh \cite{trinh2019central} points out that in the case $d=2$ there are no restrictions on the choice of $r$ as $I_2(\cP)$ is empty; this can be shown with a duality property. Finally, regarding (iii), Yogeshwaran et al. \cite{yogeshwaran2017random} remark that the condition $r\notin I_d(\cP)$ is likely to be superfluous and, as already mentioned, we show that, indeed, the condition can be removed.

\section{Background for the strong stabilization}\label{Section_Stabilization}
In our analysis of multivariate asymptotic normality of persistent Betti numbers, stabilization properties are crucial. Kesten and Lee \cite{kesten1996central} introduced stabilization to prove asymptotic normality for the weight of the Euclidean minimal spanning tree. The concept was extended by Penrose and Yukich \cite{penrose2001central, penrose2003weak} to treat general functionals defined on Poisson and binomial point processes and goes as follows. Consider a functional $H$, which is defined on finite subsets of $\R^d$, and define its add one cost function as
\[
	\fD_0 H(\cH) \coloneqq H(\cH \cup \{0\}) - H(\cH)
	\]
for $\cH \subseteq\R^d$ finite. The functional $H$ is strongly stabilizing on the homogeneous Poisson process with intensity $\lambda\in (0,\infty)$ on $\R^d$, denoted by $\cP(\lambda)$, if there exist $a.s.$ finite random variables $S$ and $\fD_{\infty}H$ such that for all finite $A\subseteq \R^d\setminus B(0,S)$,
\[
	\fD_0 H( (\cP(\lambda)\cap B(0,S) ) \cup A ) = \fD_{\infty} H \quad a.s.
	\]
Recall that, for $n\in\N$, $B_n = [-n^{1/d}/2, n^{1/d}/2]^d$ denote observation windows, and let $\cA$ be the collection $\{B_n+x: x\in\R^d, n\in\N \}$. The functional $H$ is weakly stabilizing on $\cA$ (for $\cP(\lambda)$) if there is an $a.s.$ finite random variable $\fD'_{\infty} H$ such that, for any such sequence $(A_n: n\in\N)$ from the collection $\cA$ with $\lim_{n\to \infty}A_n = \R^d$,
\[
	\fD_0 H( \cP(\lambda) \cap A_n ) \to \fD'_{\infty} H \quad \text{a.s.}\; \text{as $n \to \infty$.}
	\]
Recall that the set-theoretic limit $\lim_{n\to \infty}A_n = \R^d$ is equivalent to $\lim_{n \to \infty}{\bf 1}_{A_n}(x) =1 $ for all $x \in \R^d$. The stabilization of a functional defined on subsets of a point process roughly means that a local change in the point process (e.g., adding or subtracting finitely many points) affects the value of the functional only locally. This latter phenomenon can be described with different notions. We consider two radii of stabilization for the persistent Betti function $\beta^{r.s}_q$. Their functionality is related to the classical weak and the strong stabilization property given above. Properties of these radii are addressed below in detail.

Consider a point process $P$ on $\R^d$ without accumulation points and let $Q$ be finite with circumcenter $z_Q\in\R^d$ and circumradius $L_Q$. So, $Q \subseteq B(z_Q,L_Q)$ for $L_Q\ge 0$ minimal. For short, write $\cK_{r,a} =  \cK_r( P \cap B(z_Q,a) )$ and $\cK'_{r,a}= \cK_r( (P\cup Q) \cap B(z_Q,a) )$ for $a,r \ge 0$. In the following, the reference case is that $Q\subset Q(0)$, so that $z_Q\in Q(0)$ and $L_Q \le \sqrt{d}/2$. (Recall that $Q(0)$ is defined in Section~\ref{Section_DefinitionsNotation}.)\\[5pt]
\noindent\textbf{Radius of weak stabilization:} Define the radius of weak stabilization of $(r,s)$ by
\begin{align}
		\rho^{(r,s)}_q (P,Q) &\coloneqq \inf\{	R>0: \dim Z_q(\cK'_{r,a}) - \dim Z_q(\cK_{r,a}) = \text{const. } \forall a\ge R \text{ and } \nonumber \\
		&\quad \dim Z_q(\cK'_{r,a})\cap B_q(\cK'_{s,a}) - \dim Z_q(\cK_{r,a})\cap B_q(\cK_{s,a}) = \text{const. }\; 
		\forall a\ge R\} \nonumber \\
		\rho^{(r,s)}(P,Q) &\coloneqq \max_{0\le q\le d-1} \rho_q^{(r,s)}(P,Q). \label{Def:WeakStabilizationRadius}
\end{align} 
One can use similar ideas as in the proof of Lemma 5.3 in Hiraoka et al.\ \cite{hiraoka2018limit} to show that if $P$ has $a.s.$ no accumulation points and if $Q$ is finite, then $\rho^{(r,s)}(P,Q)$ is $a.s.$ finite; we do this in Lemma~\ref{L:MeaningfulRho} in the Appendix. A similar result was also obtained by Hiraoka et al.\ \cite{hiraoka2018limit} for the add one cost function of persistent Betti numbers. Our definition of the radius of weak stabilization implies that for all $0\le r \le s$ and for all $ q\in \{0,\ldots,d-1\},$
\[
			\beta^{r,s}_q\Big( \cK\big( (P\cup Q)\cap B(z_Q,R) \big) \Big) - \beta^{r,s}_q\Big( \cK\big( P\cap B(z_Q,R) \big) \Big) = \text{const.}
\]
as a function in $R$, for  $R\ge \rho_q^{(r,s)}(P,Q).$\\ 

\noindent\textbf{Radius of strong stabilization:} Let $r>0$ be an arbitrary but fixed filtration parameter. Let $\mu(r)$ be an upper bound on the diameter of simplices in the filtration at time $r$. For the Vietoris-Rips filtration, $\mu(r)$ equals $r$. For the {\v C}ech filtration $\mu(r)=2r$ is a sharp bound. We choose $a\ge a^*(r) = L_Q + \mu(r)$ sufficiently large such that all simplices containing at least one point of $Q$ have a filtration time of at most $a^*(r)$; recall that $L_Q$ denotes the circumradius of $Q$.

Given $P$ and $Q$, let $\sigma^r_{q,i}$, $i=1,\ldots,m_q$, be the $q$-simplices in $\cK'_{r,a}\setminus \cK_{r,a}$ contained in the ball $B(z_Q,a)$ that are created until filtration time $r$ due to the addition of the points in $Q$ to the point process $P$. W.l.o.g., the simplices are already ordered according to their filtration time; simplices with the same filtration time are ordered at random.

We call the number $R$ that limits the knowledge of a point process $P'$ to the ball $B(z,R)$, the information horizon (w.r.t.\ $z$), i.e., we only observe the process $P'\cap B(z,R) = P'|_{B(z,R)}$ and the corresponding simplicial complexes restricted to $P'|_{B(z,R)}$, i.e., the complexes $\cK_r( P'|_{B(z,R)} )$, $r\ge 0$.

Let $\partial_q$ denote the $q$-th boundary map. For $i=1,\ldots,m_q$, define the following quantities depending on the filtration parameter $r\ge 0:$
\[
			C^{\,r}_{q,i}(a) = C_q ( \cK_{r,a})  \oplus \langle \sigma^r_{q,1},\ldots,\sigma^r_{q,i} \rangle \quad \text{ and } \quad Z^r_{q,i}(a) = \operatorname{ker}(\partial_q\colon C^{\,r}_{q,i}(a)  \to B^r_{q-1,i}(a) ),
\]
where $B^r_{q-1,i}(a) = \partial_q(C^{\,r}_{q,i}(a) ) \subseteq C^{\,r}_{q-1,i}(a)$ is the image of $\partial_q$. First, we define
\begin{align}
		\wt\rho_q^{\,r}(P,Q) \coloneqq \inf\Big\{&	R \ge a^*(r) \ \Big| \ \text{for each } \sigma^r_{q,i}, i \in \{1,\ldots,m_q \}: \nonumber\\
					& \text{ either }  \Big[ \exists c\in Z^r_{q,i}(R) : \sigma^r_{q,i}\ \text{ is contained in } c \Big] \nonumber \\		
		&  \text{ or }  \Big[ \text{ conditional on  $(P\cup Q) |_{B(z_Q,R)}$ } \nonumber \\
		 &  \big[ \forall a\ge R, \forall c\in Z^r_{q,i}(a): \sigma^r_{q,i} \text{ is not contained in $c$ } \big] \text{ is true } \Big] \Big \}. \label{Def:StrongStabilizationRadius1}
\end{align}
This definition means the following: Consider an information horizon $R  > \wt\rho_q^{\,r}(P,Q)$, i.e., we observe all points from the point process $(P\cup Q)\cap B(z_Q,R)$. Then if we include the $q$-simplex $\sigma^r_{q,i}$ in the simplicial complex, we already have the information that either $\sigma^r_{q,i}$ creates a new $q$-cycle or that it does not, even when additionally including points from an infinite information horizon. In other words, given $(P\cup Q)\cap B(z_Q,R)$, the event which simplices $\sigma^r_{q,i}$ are ultimately positive is decidable (i.e.\ computable or recursive). By ``a simplex being ultimately positive'' we mean that if the information horizon is large enough, then we see that the simplex is part of a cycle. Similarly, a simplex staying negative means that it never becomes part of a cycle, even if the information horizon is infinite.

Thus, the event of which simplices $\sigma^r_{q,i}$ are ultimately positive being decidable means that having observed $P \cup Q$ up to the information horizon $R$, i.e., $(P \cup Q)\cap B(z_Q,R)$, where $R$ is `large enough' ($ R > \wt\rho_q^{\,r}(P,Q)$), we know that each potential cycle in $C_q( \cK_r( \cP(\lambda)|_{B(z_Q,R)} ) ) \oplus \langle \sigma^r_{q,1},\ldots,\sigma^r_{q,i} \rangle$ containing $\sigma^r_{q,i}$ has already terminated, meaning that $\sigma^r_{q,i}$ is positive, or that $\sigma^r_{q,i}$ will stay negative.

In order to state, whether the persistent Betti number for a given pair $(r,s)$ and a given dimension $q$ has stabilized, we need to know whether a new cycle in $Z_q (\cK'_{r,\wt\rho^{\,r}_q})$ is also a $q$-dimensional feature, i.e., is not eventually a boundary in $B_q (\cK'_{s,a})\cap Z_q (\cK'_{r,a})$ for some $a>\wt\rho^{\,r}_q$.

For this purpose, apart from the additional $q$-simplices $\sigma^r_{q,1},\ldots,\sigma^r_{q,m_q}$, we write $\wt\sigma^s_{q+1,1},\ldots,\wt\sigma^s_{q+1,\wt m_{q+1}}$ for the additional $(q+1)$-simplices in $\cK'_{s,a}\setminus \cK_{s,a}$ for $a \ge a^*(s)$. 

We can repeat the considerations from above for the $(q+1)$-dimensional simplices and a filtration parameter equal to $s$. Then by definition, after time $\wt\rho^{\,s}_{q+1}(P,Q)$, we know for each $\wt\sigma^s_{q+1,i}$ whether it is positive or whether it is negative for all $a>\wt\rho^{\,s}_{q+1}(P,Q)$. 

Consequently, the radius of strong stabilization for the pair $(r,s)\in\Delta$ is as follows.
\begin{align}
		S_q^{(r,s)}(P,Q) = \max\{ \wt\rho^{\,r}_q(P,Q), \wt\rho^{\,s}_{q+1}(P,Q) \}.\label{Def:StrongStabilizationRadius2}
\end{align}
At this stage there is a major difference between the \v Cech and the Vietoris-Rips complex if $q=d-1$. For the \v Cech filtration, there are no $q$-dimensional cycles in $d$-dimensional Euclidean space for $q\ge d$. Thus, for the \v Cech filtration $S^{(r,s)}_{d-1}(P,Q) = \max\{\wt\rho^{\,r}_{d-1}(P,Q), a^*(s)\}$. For the Vietoris-Rips filtration, however, there can be $q$-dimensional cycles for every possible dimension $q$, we refer to Bobrowski and Kahle \cite{bobrowski2014topology}. In particular, $S^{(r,s)}_{d-1}(P,Q) >  \max\{\wt\rho^{\,r}_{d-1}(P,Q), a^*(s)\}$ can occur.

In the following, if $Q=\{0\}$, we simply write $\rho^{(r,s)}(P)$, $\wt\rho^{\,r}_q(P)$ or $S^{(r,s)}_q(P)$ for convenience. Next, we show in Theorem~\ref{Thrm:StrongStabilization} that the radius $\wt\rho_q^{\,r}(P,Q)$ is $a.s.$ finite for each $q\ge 0$ and $r\in\R_+$ if $P$ equals a homogeneous Poisson process modulo a finite set of points and $Q\subseteq\R^d$ is finite. In particular, this implies the strong stabilization property of the persistent Betti number $\beta^{r,s}_q$ in the sense, that $S^{(r,s)}_q$ from \eqref{Def:StrongStabilizationRadius2} is finite which in turn leads to Theorem~\ref{T:StrongStabAddOne}.

\begin{theorem}\label{Thrm:StrongStabilization}
For a Poisson process with constant intensity $\lambda\in\R_+$ and two finite (disjoint) sets $Q_1,Q_2\subseteq Q(0)$, the radius $	\wt\rho_q^{\,r}(\cP(\lambda)\cup Q_1,Q_2)$ is $a.s.$ finite for each $q$ and for each $r>0$. In particular, the radius of strong stabilization $S_q^{(r,s)}$ is finite for each $q$ and each $(r,s) \in \Delta$.
\end{theorem}

We apply arguments from continuum percolation theory for the proof of this theorem. These arguments are also used to prove uniqueness of the occupied and vacant component in the Boolean model, see e.g., Aizenman et al.\ \cite{aizenman1987uniquenessBook, aizenman1987uniqueness}, Burton and Keane \cite{burton1989density}, as well as the monograph of Meester and Roy \cite{meester1996continuum}.

Furthermore, we have the following relation between the two radii.
\begin{lemma}\label{L:StrongAndWeakStabilization}
$\rho^{(r,s)}_q(P,Q) \le  S_q^{(r,s)}(P,Q)$ for each pair $(r,s)\in \Delta$, $q\in\{0,\ldots,d-1\}$. 
\end{lemma}
Thus, weak stabilization measured in terms on $\rho^{(r,s)}_q$ is always implied by strong stabilization measured in terms of $S_q^{(r,s)}$. The proof of this lemma is given in Subsection~\ref{Subsection_Stabilization}.

Moreover, our results are not limited to this \textit{static} case, where we only consider one Poisson process and the persistent Betti number for one pair $(r,s)$: We show in Theorem~\ref{Thrm:StrongStabilizationAppl} that Borel probability measures induced by the radius of strong (and of weak) stabilization are tight over a variety of parameter ranges.

The theorem is divided into three parts. Part (1) considers uniform stabilization over a variety of homogeneous Poisson processes. These stabilization properties then enable us to derive the results in parts (2) and (3), where we consider stabilization for the binomial and the Poisson sampling schemes, respectively. The proof of Theorem~\ref{Thrm:StrongStabilizationAppl} is deferred to the Appendix~\ref{AppendixTightnessOfStabilization}.

\begin{theorem}[Uniform stabilization]\label{Thrm:StrongStabilizationAppl}
For $m \in \mathbb N$, let $\fQ_m = \{ \{y_1,\ldots, y_k\}: y_i\in Q(0), i=1,\ldots,k, k\le m\}$ be the class of sets with at most $m$ points in $Q(0)$. Let $+\infty> \ol{r}\ge \ul{r}>0$ and $m\in\N$ be arbitrary but fixed. Then we have the following:

\begin{itemize}\setlength\itemsep{0em}

\item [(1)] \textbf{Stabilization for the homogeneous Poisson case:} The laws of 
\[
\{ \wt\rho_q^{\,r} (\cP(\lambda)\cup Q_1,Q_2)): \ul{r}\le r \le \ol{r},\lambda\in \R_+, Q_1,Q_2\in \fQ_{m}, q=0,\ldots,d-1 \}
\]
are tight for each $m\in\N$.

\item [(2)]  \textbf{Stabilization in the Poisson sampling scheme:} Let $\nu$ be a probability density on $[0,1]^d$. For $n \in \N$ and $L>0$, set $B''_{n,L} = \{z\in\R^d: B(z,L) \subseteq [0,n^{1/d}]^d \}$. Consider a specific continuous density $\kappa$ on $[0,1]^d$.

Let $\epsilon>0$. Then there are $b>0$, $n_0\in\N$ and $L\in\R_+$ such that, uniformly in $q=0,\ldots,d-1$ and $r\in[\ul{r},\ol{r}],$
\[
			\sup_{n\ge n_0} \quad  \sup_{z \in B''_{n,L}} \quad  \sup_{Q_1,Q_2 \in z + \fQ_{m} } \quad \p( \wt\rho_q^{\,r} (n^{1/d} \cP(n \nu) \cup Q_1, Q_2 ) \ge L ) \le \epsilon
\]
for all densities $\nu$ on $[0,1]^d$ satisfying $ \| \nu - \kappa\|_\infty \le b$.

For each $n\in\N$, let $\cV_n, \cW_n$ be Poisson processes on $[0,n^{1/d}]^d$ that are independent of $n^{1/d}\cP(n\nu)$, and whose intensity functions on $\R^d$ are uniformly bounded in $n$. Then for each $\epsilon>0$, there are $b>0$, $n_0\in\N$ and $L>0$ such that, uniformly in $q=0,\ldots,d-1$ and $r\in[\ul{r},\ol{r}],$
\[
	  \sup_{n\ge n_0} \quad  \sup_{z \in B''_{n,L}} \quad  \p( \wt\rho_q^{\,r}(n^{1/d} \cP(n \nu) \cup ( \cV_n  \cap Q(z) ), \cW_n  \cap Q(z) ) \ge L ) \le \epsilon
\]
for all densities $\nu$ on $[0,1]^d$ satisfying $\| \nu - \kappa \|_{\infty} \le b$.

\item [(3)] \textbf{Stabilization in the binomial sampling scheme:}  Let $\mX_n$ be an $n$-binomial process on $[0,1]^d$ obtained from an i.i.d.\ sequence $(X_k:k\in\N)$ with common density $\kappa$. Let $X'$ be a random variable  independent of $(X_k:k\in\N)$, and with continuous density $\kappa$ on $[0,1]^d$. Write $Q_{m,n}$ for the point process $n^{1/d}(\mX_m - X')$ for $m\in J_n=[n-h(n),n+h(n)]$, where the function $h$ satisfies $h(n)\rightarrow\infty$ and $h(n)/n\rightarrow 0$ as $n\rightarrow\infty$. Then the family $\{ \wt\rho_q^{\,r}(Q_{m,n},\{0\}): n\in\N,m\in J_n, \ul{r}\le r\le \ol{r}, q=0,\ldots,d-1\}$ is tight.
\end{itemize}
Furthermore, with $\wt\Delta = \{(r,s)\in\Delta: w_1\le r\le s\le w_2\}$, where $w_2\ge w_1 >0$ are arbitrary, all these results remain valid if $\wt\rho_q^{\,r}$, $\ul r \le r\le \ol{r}$ is replaced by $\rho^{(r,s)}_q, (r,s) \in \wt\Delta$.
\end{theorem}

\section{Technical results}\label{Section_TechnicalResults}
This section consists of three parts. We derive the stabilization results in the first part. Then we prove the asymptotic normality of the finite-dimensional distributions in case of underlying Poisson processes in the second part. In the third part we prove the same limit result in case of an underlying sequence of binomial processes.

The next result is crucial for the upcoming proofs. It is a direct consequence of Lemma 2.11 in Hiraoka et al. \cite{hiraoka2018limit}. This so-called geometric lemma, enables us to obtain upper bounds on moments.  
\begin{lemma}[Corollary of \cite{hiraoka2018limit} Lemma 2.11] \label{L:GeometricLemma}
Let $\mX \subseteq \mY$ be two finite point sets of $\R^d$. Then
\[
		\left|\beta^{r,s}_q (\cK(\mY)) - \beta^{r,s}_q (\cK(\mX))	\right| \le \sum_{j=q}^{q+1}  |\cK_j(\mY,s) \setminus \cK_j(\mX,s)|.
\]
\end{lemma}

\subsection{Stabilization}\label{Subsection_Stabilization}
We start with the proof of the fundamental Theorem~\ref{Thrm:StrongStabilization}; it uses Proposition~\ref{P:Encounters} in its last step. While this proposition is stated immediately after the proof of Theorem~\ref{Thrm:StrongStabilization}, it is of course helpful to first read this proposition  before passing to this last step.

Subsequently, we derive Theorem~\ref{T:StrongStabAddOne} and Lemma~\ref{L:StrongAndWeakStabilization}. The proof of Theorem~\ref{Thrm:StrongStabilizationAppl} is deferred to the Appendix~\ref{AppendixTightnessOfStabilization}.

\begin{proof}[Proof of Theorem~\ref{Thrm:StrongStabilization}]
If $q=0$, $\wt\rho_0^{\,r}$ is clearly finite. So we can assume that $q>0$ and we have to study chains that prevent $\wt\rho_q^{\,r}(\cP(\lambda)\cup Q_1,Q_2)$ from being finite. We can consider $\cP = \cP(1)$ because we consider a general positive $r$. Moreover, we can assume that the circumcenter $z_{Q_2}$ of $Q_2$ coincides with the origin. We write $a^*(r) = L_{Q_2} + \mu(r)$ with $L_{Q_2}$ being the circumradius of $Q_2$. (Here, $\mu(r) = r$ for the Vietoris-Rips and $\mu(r)=2r$ for the \v Cech filtration.) Clearly, as $Q_1,Q_2$ are finite, $\wt\rho_q^{\,r}$ is finite if the random geometric graph $G(\cP,\mu(r))$ does not percolate. If $G(\cP,\mu(r))$ percolates, $\wt\rho_q^{\,r}$ is infinite if and only if for each finite information horizon $R\in\R_+$, $R \ge a^*(r)+\mu(r)$, there is a simplex $\sigma_{q,i}\in \cK_r( (\cP\cup Q_1 \cup Q_2)|_{B(0,R)})$, which intersects with the additional points $Q_2$ and which is negative until $R$, but we cannot exclude the possibility that it might become positive ultimately. 
Formally, this means there is a chain 
\begin{align}\label{E:StrongStabilization0}
	\tau = \sum_{i} \sigma_i,
\end{align}
where $\sigma_i \in \bigcup_{n\in\N} \cK_r( (\cP\cup Q_1)|_{B(0,n)})$
are $q$-simplices, such that the boundary of the restriction of $\tau$ to $\cK_r( (\cP\cup Q_1)|_{B(0,R)})$ consists of two disjoint $(q-1)$-cycles which are not boundaries. More precisely, set
\begin{align}\label{E:StrongStabilization1}
	\tau_R := \tau|_{B(0,R)} := \sum_{i} \sigma_i\1{\sigma_i \in \cK_r((\cP\cup Q_1)|_{B(0,R)}) }.
	\end{align}
Then, for each $R\ge a^*(r)+\mu(r)$, we have $\partial \tau_R = e_1 + e_{2,R}$, where
\begin{align}
	\begin{split}\label{E:StrongStabilization1b}
		& e_1,e_{2,R} \in Z_{q-1}( \cK_r( (\cP\cup Q_1)|_{B(0,R)})) \setminus B_{q-1}( \cK_r( (\cP\cup Q_1)|_{B(0,R)})) \\
		&\qquad\text{such that } e_{2,R} \subseteq B(0,R)\setminus B(0,R-2\mu(r)) \\
		&\qquad\text{ and $e_1 \subseteq B(0,R_0)$ for a certain $R_0\in \R_+$};
	\end{split}
\end{align} 
here the set inclusions for $e_{2,R}$, resp., $e_1$ are to be understood as usual inclusions between subsets of the Euclidean space.
The cycle $e_1$ becomes a boundary of $\tau|_{B(0,R)}$ when including the additional points of $Q_2$, i.e., $e_1 = \sum_j \nu_j$, where the $\nu_j$ are $(q-1)$-simplices in $\cK_r((\cP\cup Q_1\cup Q_2)|_{B(0,a^*(r)+\mu(r))})$ and $e_1 \in B_{q-1}( \cK_r((\cP\cup Q_1\cup Q_2)|_{B(0,a^*(r)+\mu(r))}) )$. Consequently, $e_{2,R}$ becomes a boundary in this case as well, i.e., $e_{2,R} \in B_{q-1}(\cK_r((\cP\cup Q_1\cup Q_2)|_{B(0,R)}))$.  See also Figure~\ref{fig:TubeSituation_Illustration} for an illustration in the special case of such a 1-dimensional chain.

The existence of a $\tau$ as above is equivalent to $\wt\rho_q^{\,r}$ being infinite. We show in the remainder of the proof: For a homogeneous Poisson process modulo a finite point process such chains cannot occur. More precisely, we show that the cycle $e_{2,R}$ cannot exist for all $R\ge a^*(r)+\mu(r)$. For this purpose, we can assume that $Q_1=\emptyset$ because the question whether such cycles $e_{2,R}$ exist for all $R\ge a^*(r)+\mu(r)$ is an asymptotic property of the Poisson process (this follows also from details given below). Moreover, it suffices to study the case where $Q_2$ consists of a single point and by translation invariance we can assume that this is the origin, viz., $Q_2=\{0\}$.

In the remainder of the proof, we study chains which generalize the chain $\tau$ given in \eqref{E:StrongStabilization0} and show with a Burton-Keane argument that the existence of such chains contradicts the properties of the stationary Poisson process. To this end, we proceed in four steps. We introduce general {\it maximal} chains in the first step. In the second step, we show that the number of these maximal chains is a.s.\ constant. In the third and fourth step, we prove that the number of these maximal chains is a.s. zero.

{\it Step 1 -- maximal chains.} We begin with $q$-chains $\tau$ of the following type: $\tau=\sum_{i}\sigma_i$, where the simplices $\sigma_i$ are in $\bigcup_{n\in\N} \cK_r( \cP|_{B(0,n)})$ for all $i$, and there are $y\in\R^d$ (`a shift of the origin') and $R_0\in\R_+$ such that the following property (P) holds:
\begin{figure}[htbp]
	\centering
		\includegraphics[width=1.00\textwidth, trim=3cm 2.5cm 3cm 2cm, clip=true]{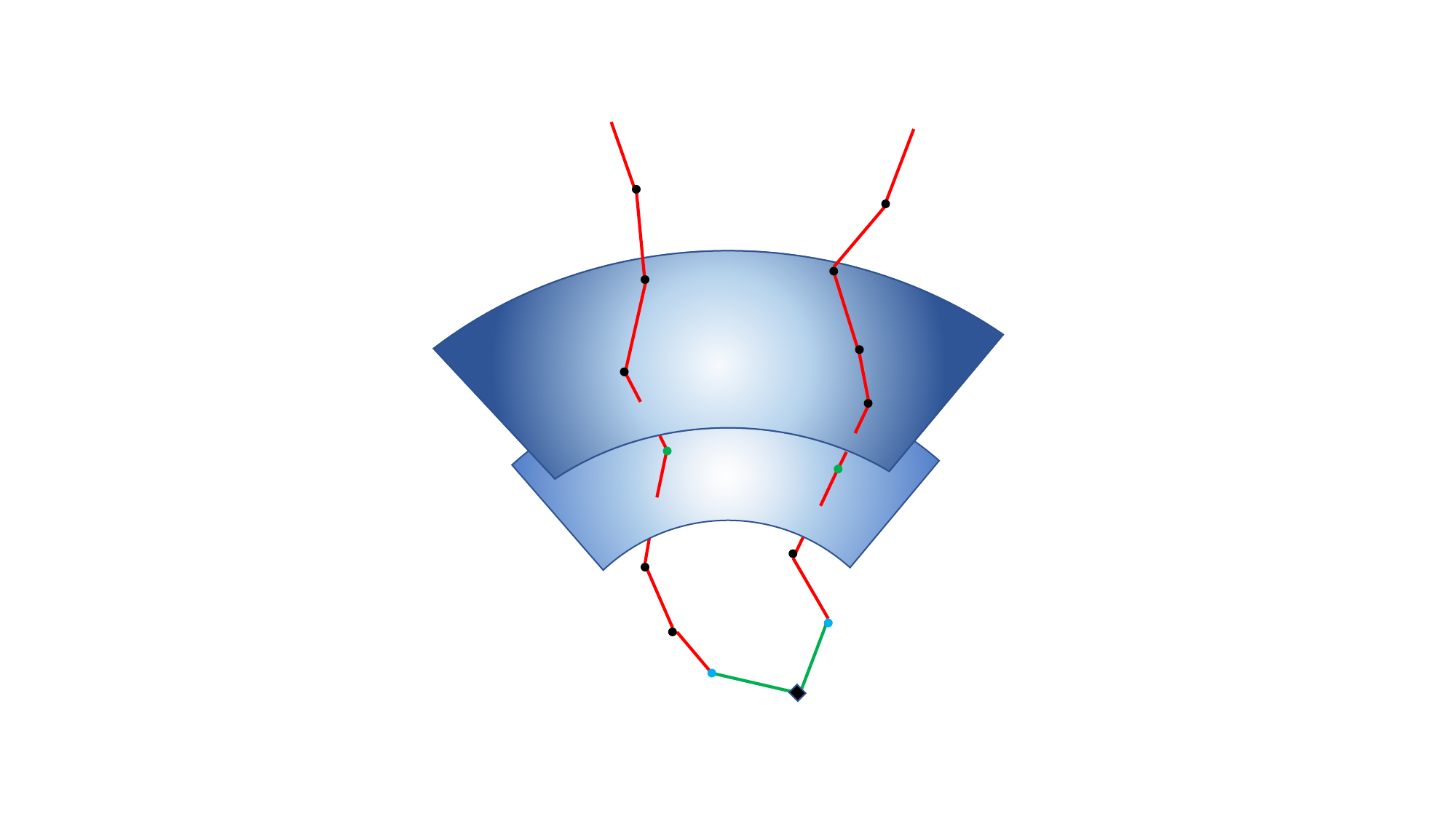}
	\caption{Illustration of a chain $\tau$ consisting of 1-dimensional simplices (red, green) from Poisson points (black, blue and green dots) and an additional point (black diamond) which is located inside $Q(0)$. The 1-simplices between Poisson points are red, the 1-simplices between a Poisson point and the additional point are green. The layers depict two spheres of $B(z,R )$ and $B(z,R-2\mu(r))$. $e_1$ corresponds to the two blue dots (to which the two green 1-simplices are attached), $e_{2,R}$ to the green dots shown between the two layers}.
	\label{fig:TubeSituation_Illustration}
\end{figure}
\begin{description}
\item {\bf (P)} For all $R\ge R_0$ and with
$$
	\tau_{y,R} := \sum_{i} \sigma_i \1{ \sigma_i\in \cK_r( \cP|_{B(y,R)} )},
$$
we have that $\partial \tau_{y,R} = e_1 + e_{2,R}$, where
\begin{align}
\begin{split}\label{E:StrongStabilization2}
	 e_1 \in & Z_{q-1}( \cK_r( \cP|_{B(y,R)})) \setminus B_{q-1}( \cK_r( \cP|_{B(y,R)})) \\
	&\text{ but } e_{1} \in B_{q-1}( \cK_r( (\cP\cup \{y\} )|_{B(y,a^*(r)+\mu(r))}))
	\end{split}
 \shortintertext{and}
\begin{split}\label{E:StrongStabilization3}
	e_{2,R} &\subseteq B(y,R)\setminus B(y,R-2\mu(r)) \\
	&\text{ and } e_{2,R} \in Z_{q-1}( \cK_r( \cP|_{B(y,R)})) \setminus B_{q-1}( \cK_r( \cP|_{B(y,R)}).
\end{split}
\end{align}
\end{description}
\vspace*{0.3cm}
So loosely speaking, $\tau_{y,R}$ is a shifted analogue of $\tau_R$ from \eqref{E:StrongStabilization1}.
We call a chain $\tau$ satisfying (P) a {\em chain of type (P)}.

Now, we might have two chains $\tau$ and $\tau'$ of type (P) that have the same local cycle, i.e., their cycles $e_1, e'_1$ (as characterized by \eqref{E:StrongStabilization2}) are homologous (i.e., they differ by a boundary only). We say that $\tau$ and $\tau'$ are equivalent and write $\tau\sim\tau'$.
In this case, we can consider the chain $\tau^*$ which consists of the union of the simplices of $\tau$ and $\tau'$.

This leads to the following notion of maximality: Consider an arbitrary but fixed chain $\tau$ of type (P) with local cycle $e_1$. Then take the union over all chains $\tau'$ of type (P) equivalent to $\tau$ and call the resulting chain the maximal chain $\tau_{\max}$. Formally, given $\tau$ with corresponding $e_1,$ the maximal chain is
$$
	\tau_{max} = \sum_{\sigma \in I} \sigma, \quad \text{ where } \quad I = \bigcup_{\substack{\tau':
	\tau' \sim \tau} } \, \bigcup_{\sigma\in\tau'} \{\sigma\}. 
$$
Plainly, if $\tau$ and $\tau'$ are equivalent, then $\tau_{\max} = \tau'_{\max}$.

{\it Step 2 -- ergodicity and invariance of maximal chains.} We show that the number of maximal chains, which we denote by $M(\cP)$, is almost surely constant. More precisely, we show that
$$
	M(\cP) = m \ a.s. \qquad\text{ for some } m \in \N \cup \{0, +\infty\}.
$$
To this end, consider the standard decomposition of the stationary Poisson process $\cP$ into a countable sum of independent Poisson processes restricted to the cubes $Q_z$, viz.,
\begin{align}\label{E:DecompPoisson}
	\cP = \sum_{z\in\Z^d} \cP|_{Q_z} = \sum_{n\in\Z} \bigg\{ \sum_{z'\in\Z^{d-1}} \cP|_{Q_{(n,z')}} \bigg\}.
\end{align}
Next, for $m\in\N\cup\{0,+\infty\},$ define the events
$$
	A_{m} \coloneqq \{ \omega\in\Omega \ | \ M( \cP(\omega)) = m \}.
$$
Consider the shift operator $T$ acting on the first coordinate only. Applying $T$ to a set $P\subseteq\R^d$ gives $T(P) = \{ y+ (1,0,\ldots,0)^t | y\in P\}$.
Define the ``$k$th translation of $A_{m}$'' by
$$
	T^k(A_{m} ) \coloneqq \{ \omega\in\Omega \ | \ M( T^k(\cP(\omega))) = m \}, \quad k\in\Z.
$$
Then $M(\cP) = M( T(\cP))$ and consequently, $T(A_{m}) = A_{m}$ for each $m$, i.e., $A_m$ is an invariant event. A standard argument which relies on the decomposition \eqref{E:DecompPoisson} into iid random variables now yields that $\p(A_{m}) \in \{0,1\}$ for each $m$. We refer to the book of Klenke \cite{klenke2013probability}, Example 20.26, for a formal proof.

{\it Step 3 -- insertion tolerance, $M(\cP) \notin\N$ a.s.} Assume that $M(\cP) = m$ with probability 1 for some $m\in\N$. Define the generic ``annulus'' (w.r.t.\ the maximum norm)
$A_{t,s} = [-t,t ]^d \setminus (-s,s)^d$ for $s<t$. 

We rely on the following decomposition of $\cP$: We fix $n\in\N$ large enough - see the paragraph below \eqref{E:SizeN} for details. Denote the restriction of $\cP$ to $Q(0,n)=[-n,n]^d$ by $\cPc_n$ and its restriction to $\R^d\setminus [-n,n]^d$ by $\cPd_n$. Then $\cPc_n$ and $\cPd_n$ are independent. Assume $\cPc_n$ is defined on the generic probability space $(\Omega^\circ,\cA^\circ,\p^\circ)$ and $\cPd_n$ is defined on $(\Omega^\dagger,\cA^\dagger,\p^\dagger)$.
Next, for $n\in\N$ and $\epsilon \in (0,\mu(r))$, consider the event
\begin{align}
	D_{n,\epsilon}	& = \Big\{\text{for each $\tau_{max}$: $\partial (\tau_{max}|_{A_{n+\mu(r),n} })$ contains a cycle $c$ inside $A_{n+\mu(r)-\epsilon,n}$ with } \nonumber\\
	&\qquad \text{$c\in Z_{q-1}( \cK_r( \cP|_{Q(0,n+\mu(r))})) \setminus B_{q-1}( \cK_r( \cP|_{Q(0,n+\mu(r))})$}\Big\}. \label{E:SizeN}
\end{align}
(Here $\tau_{max}|_{A_{n+\mu(r),n} }$ is the restriction of $\tau_{max}$ to  $A_{n+\mu(r),n}$.)
We choose $n\in\N$ sufficiently large such that $D_{n,0}$ occurs with positive probability $\eta>0$.
Next, define the event 
$$
	E_{n,\epsilon} = \{ \omegad \in \Omega^\dagger \ | \ \exists\, \wt\omega^\circ \in \Omega^\circ : ( \wt\omega^\circ,\omegad) \in D_{n,\epsilon}  \}.
	$$
Note that $D_{n,\epsilon}$, and thus also $E_{n,\epsilon}$, is decreasing in $\epsilon$,  and clearly, $\p^\dagger(E_{n,0})>0$ as well. Moreover, $E_{n,0} \setminus E_{n,\epsilon}\downarrow N$ for a $\p^\dagger$-nullset $N$ as $\epsilon\downarrow 0$; hence, we can fix some $\epsilon^*>0$ sufficiently small such that $\p^\dagger(E_{n,\epsilon^*})>0$.

For more technical details underlying the following remaining argument of Step 3, we refer to the proof of Theorem~\ref{Thrm:StrongStabilizationAppl} (1). First, partition $A_{n,n-\delta}$ with subcubes $(C_i)_{i\in I}$ of edge length $\delta$ for $\delta \le \mu(r)(1-1/\sqrt{2})/\sqrt{d}$. We can assume that $n/\delta\in\N$. Let $G_n$ denote the event that each $C_i$ contains at least $d$ points of $\cPc_n$. Then 
$\p^\circ(G_n)>0$, and by construction, $E_{n,\epsilon^*}$ and $G_n$ are independent.
Consequently, $\p^\circ\otimes\p^\dagger(E_{n,\epsilon^*} \cap G_n ) = \p^\circ(E_{n,\epsilon^*})\p^\dagger(G_n) > 0$. However, if both $E_{n,\epsilon^*}$ and $G_n$ occur, there are no maximal chains because each potential feature associated to some $\tau_{max}$ already terminates in $A_{n,n-\delta}.$ This is because the points of the Poisson process are sufficiently dense inside $A_{n,n-\delta}.$  Hence, we arrive at a contradiction, and thus $M(\cP) \notin\N$ a.s.

{\it Step 4 -- the Burton-Keane argument, $M(\cP)\neq +\infty$ a.s.} We begin with general considerations for the volume-boundary argument applied to cycles. The $d$-dimensional Lebesgue measure of $A_{s+\Delta,s}$ is $
	2^d \big( (s+\Delta )^d - s ^d \big) = 2^d \Delta s^{d-1} + o\big( s^{d-1} \big)$ if $s\to\infty$ and if $\Delta$ is bounded above by a constant.
 
Moreover, we have from the definition of the {\v C}ech and Vietoris-Rips filtration over a point cloud: For each $(q-1)$-cycle $\wt e$ that is contained in the complex corresponding to the filtration parameter $r$ and that is not a boundary (i.e., a non-trivial cycle), there exists a convex set $\cJ_r$ intersecting the convex hull of $\wt e$ such that on the one hand $\cJ_r \cap \cP = \emptyset$ and on the other hand the $d$-dimensional volume of $\cJ_r$ is positive and bounded away from zero, i.e., $V(\cJ_r)\ge \delta_r$ for some $\delta_r>0$ depending on the filtration type but not on the cycle $\wt e$. Hence, the total number of such cycles that are not boundaries but are located in $A_{s+\Delta,s}$ is of order
\begin{align}\label{E:StrongStabilization5}
		\frac{2^d \Delta s^{d-1}  + o\big( s^{d-1} \big)}{V(\cJ_r)} = O(s^{d-1})
\end{align}
for $s\to\infty$ and $\Delta$ bounded above.

The remainder of this step follows very similar ideas as in the proof of Theorem~4.6 in Meester and Roy \cite{meester1996continuum} where it is shown that the number of vacant components in the standard Boolean model is a.s.\ not equal to $\infty$. To this end, we rely on a similar notation to facilitate the comparison. Also, the remainder of this step crucially relies on Proposition~\ref{P:Encounters} below, which considers the existence of what we call {\em encounter chains} (see Proposition~\ref{P:Encounters} for their definition).

It follows from Proposition~\ref{P:Encounters} that (under the assumption of infinitely many maximal chains) the event $E_m$ (detailed in this proposition) has positive probability of at least $\eta>0$, say, for all $m\in\N$ sufficiently large. So, we assume that $\p(E_m) \ge \eta$ for all $m$ sufficiently large and show that this leads to a contradiction. We can translate $E_m$ by a vector $2mz$ for $z\in\Z^d$ and call this event $E_m^{2mz}$. Then for each $L\in\N$
\begin{align}
\begin{split}\label{E:StrongStabilization6}
	&\E{ \sum_{z\in\Z^d} \1{E^{2mz}_m \text{ occurs and } Q(2mz,m)\subseteq Q(0,Lm)  } }  \\
	&\ge \eta \cdot \#\{ z\in\Z^d : Q(2mz,m)\subseteq Q(0,Lm) \} \ge \eta (L-1)^d.
\end{split}
\end{align}	
Let $\fR$ be the set of the following encounter configurations in $Q(0,Lm)$: An encounter configuration $r$ lies in $\fR$ if and only if $ Q(2mz,m)\subseteq Q(0,Lm)$ such that $r = \tau^* |_{Q(2mz,m)}$ is an encounter configuration for some encounter chain $\tau^*$. Then $\E{\#\fR}\ge \eta (L-1)^d$.

We consider the branch $b=b_{r}^{(i)}$ for each $r\in \fR$ and $i\in\{1,2,3\}$. The boundary $\partial b$ intersected with $A_{Lm+\mu(r),Lm}$ consists of $(q-1)$-cycles $e_i$ that are not boundaries themselves. We can assume that each $e_i$ is minimal in the sense that we cannot decompose the chain $e_i$ into two or more disjoint $(q-1)$-cycles. We define $\fV_b$ as the set that contains all these minimal $(q-1)$-cycles. Further, set $\fV = \bigcup_{r\in \fR} \bigcup_{i=1}^3 \fV_{b_r^{(i)}}$ which is the union of all minimal $(q-1)$-cycles that are not boundaries and that are contained in $A_{Lm+\mu(r),Lm}$. Hence, $c L^{d-1} \ge \# \fV$ by \eqref{E:StrongStabilization5} for a suitable $c\in\R_+$.
Further, for $r\in \fR$ and $i\in\{1,2,3\},$ define  
\begin{align*}
	\fC_r^{(i)} = \{r'\in \fR: r' \subseteq b_r^{(i)} \} \cup \fV_{b_r^{(i)} }
\end{align*}
as a set of chains. Then clearly $\# \fC_r^{(i)} \ge \# \fV_{b_r^{(i)}} \ge 1 =: \fK$ and the following relation holds for each $r,r'\in \fR$: Either $\big[{r} \cup \bigcup_i \fC_r^{(i)} \big] \cap \big[{r'} \cup \bigcup_j \fC_{r'}^{(j)} \big]=\emptyset$ or there are $i,j$ which satisfy
$$
	\fC_r^{(i)}\supseteq \{r'\} \cup \bigcup_{\ell\neq j} \fC_{r'}^{(\ell)} \text{ and } \fC_{r'}^{(j)}\supseteq \{r\} \cup \bigcup_{\ell\neq i} \fC_{r}^{(\ell)}.
$$ 
Define $\fS := \fR \cup \bigcup_{r\in R} \bigcup_{i=1}^3 \fC_r^{(i)}$. Then the assumptions of Lemma~3.2 in Meester and Roy \cite{meester1996continuum} are satisfied. Consequently, $\# \fS \ge \fK ( \#\fR+2)+\#\fR$. Since $\# \fS = \# \fR + \# \fV$, it follows that $\# \fV \ge \# \fR +2$. Consequently, we arrive at the following inequalities
\begin{align*}
	c L^{d-1} \ge \E{ \# \fV } &\ge \E{ \# \fV \1{\fR \neq \emptyset} }\ge  \E{ (\# \fR + 2) \1{\fR \neq \emptyset}} \\
	&= \E{ \# \fR + 2 \1{\fR \neq \emptyset}} \ge (L-1)^d \eta
\end{align*}
for all $L\in\N$. This leads to a contradiction if $L$ is sufficiently large. Therefore, $\eta=0$ and $M(\cP)$ equals $\infty$ with probability 0. This completes the fourth step.

All in all, this contradicts the initial assumption that with positive probability the chain $\tau$ from \eqref{E:StrongStabilization0} exists and satisfies ({\bf P}). Consequently, $\wt \rho^{\,r}_q$ is $a.s.$ finite.
\end{proof}

To state and prove the upcoming proposition, we rely once more on a suitable decomposition of the homogeneous Poisson process $\cP$ with unit intensity. Fix $n\in\N$. We choose a Poisson process $\cPd_n$ on $\R^d\setminus [-n,n]^d$ which is defined on the probability space $(\Omega^\dagger,\cA^\dagger,\p^\dagger)$. And we choose a complementary Poisson process $\cPc_n$ on $Q(0,n)=[-n,n]^d$ which is defined on $(\Omega^\circ,\cA^\circ,\p^\circ)$. Obviously, $\cPd_n$ and $\cPc_n$ are independent.

\begin{proposition}[Encounters of maximal chains]\label{P:Encounters} For $m \in \N$, let  $E_m$ denote the event given in the subsequent paragraph. If the numbers of maximal chains is infinite (as it is assumed in step 4 of the proof of Theorem~\ref{Thrm:StrongStabilization}), we have that $\liminf_{m\to\infty} \p(E_m)>0$.

Define $n = m - 2 \ceil{\mu(r)}$ and decompose $\cP$ into $\cPc_n$ and $\cPd_n$ as above. Set\\[5pt]
$E_m := \Big\{ (\omegac,\omegad)\in \Omega^\circ\times \Omega^\dagger \mid \cPd_n(\omegad) $ admits disjoint chains $b^{(1)}, b^{(2)},b^{(3)}$  which fullfil 
\begin{itemize}
	\item[(a)]  there is an $\wt\omega^\circ\in\Omega^\circ$ such that each $b^{(i)}$ can be completed to a maximal chain with elements of $\cPc_n(\wt\omega^\circ)$ and $\cPd_n(\omegad)$
	\item[(b)] each $\partial(b^{(i)}|_{A_{m,n}})$ decomposes into a disjoint union of $e^{(i)}_1$ and $e^{(i)}_2$ that are $(q-1)$-cycles but not boundaries in the complex generated by $\cPd_n(\omegad)$ and $e_1^{(i)} \subseteq A_{n+\mu(r)-\epsilon,n}$;
	\item[(c)] there is a chain $r$ in the complex generated by the elements of $\cPc_n(\omegac)$ and a chain $c$ in $A_{n+\mu(r),n-\mu(r)}$ generated by the elements of $\cPc_n(\omegac) \cup \cPd_n(\omegad)$ such that $\tau := r + c + \displaystyle{\sum_{i=1}^3 b^{(i)}}$ satisfies $\partial (\tau|_{Q(0,m)} ) = \displaystyle{\sum_{i=1}^3 e_2^{(i)} \Big\}.}$
	\end{itemize}

If $E_m$ occurs, we call $Q(0,m)$ an encounter box, $Q(0,n)$ a central box, $r$ an encounter configuration, $c$ an intermediate configuration, $b^{(1)}, b^{(2)}, b^{(3)}$ branches and $\tau$ an encounter chain.

Moreover, we can translate $E_m$ over the vector $y=2mz$ ($y\in\Z^d$) and denote this event by $E_m^{y}$. If $E_m^y$ occurs, $Q(y,m)$ is an encounter box and $Q(y,n)$ a central box.
\end{proposition}

\begin{figure}[h]
	\centering
		\reflectbox{\includegraphics[width=.80\textwidth, trim=1.9cm 4.4cm 0.4cm 3cm, clip=true]{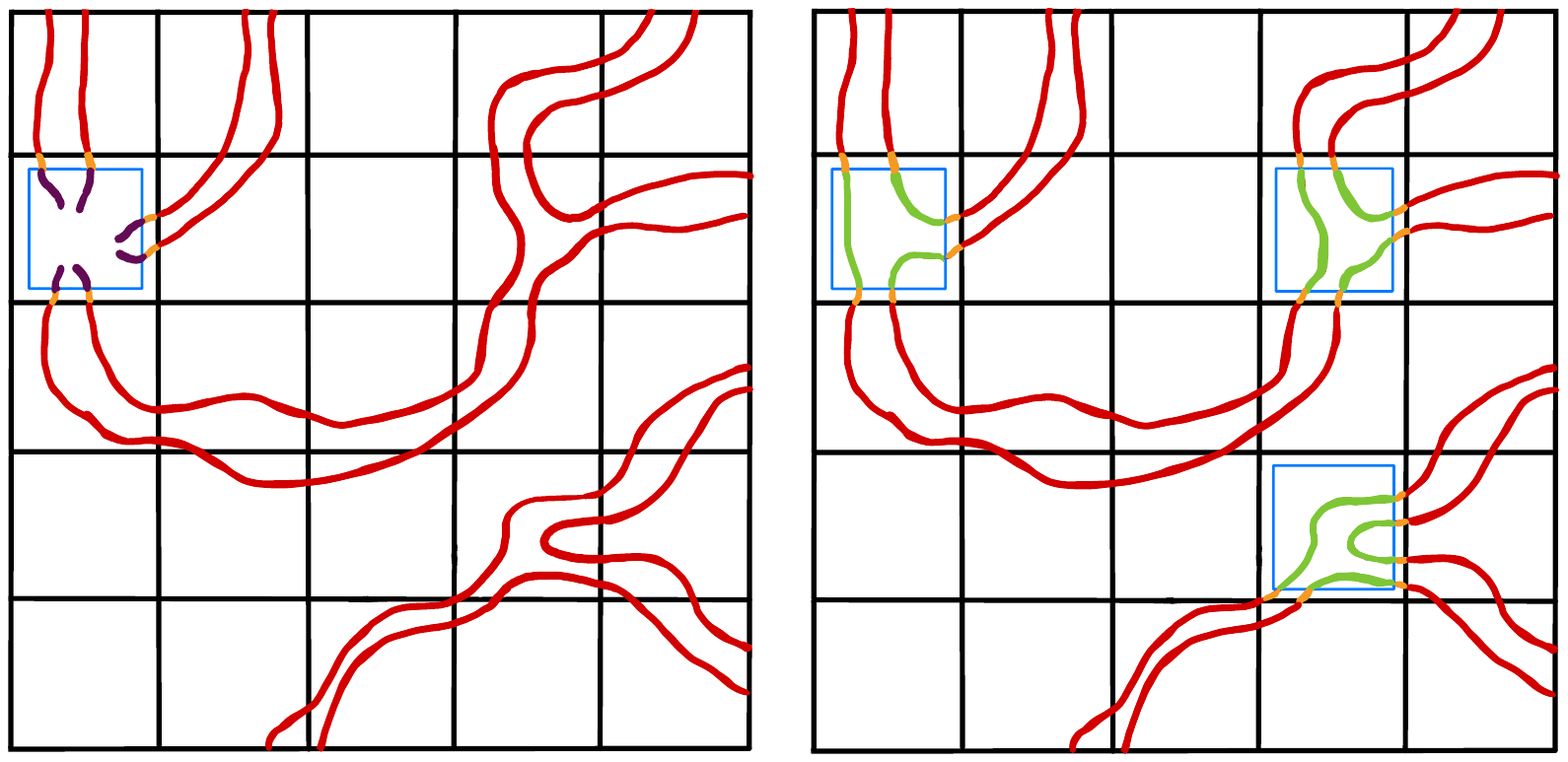}}
	\caption{Illustration of encounters of maximal chains in two dimensions (in a reduced set-up and not true to scale).
	 The left figure depicts several encounters inside boxes $Q(y,m)$ (for certain $y\in\Z^d$): The (blue) central boxes are located inside the (black) encounter boxes which are part of the (black) lattice which partitions the plane. Each (green) encounter configuration merges three branches (red and partly in green and orange) through the corresponding (orange) intermediate configuration.
	The right figure considers a specific central box $Q(y,n)$ (blue). Here a suitable configuration (violet) inside the central box converts the corresponding branches to maximal chains.}
	\label{fig:Encounters_Illustration}
\end{figure}

\begin{proof}
We assume the existence of infinitely many maximal chains with probability one and then we build consecutively an event $H_m$ which is contained in $E_m$. First let $F_m$ be the event that there are (at least) three disjoint maximal chains $\tau^{(1)},\tau^{(2)},\tau^{(3)}$ with corresponding $e^{(1)}_1,e_1^{(2)},e^{(3)}_1$ (as in \eqref{E:StrongStabilization2}) which are all located inside $Q(0,n)$, $n=m-2\ceil{\mu(r)}$. (If there are more than three such maximal chains, we choose three of them at random.) Then $\liminf_{m\to\infty} \p(F_m)=1>0$ by monotone convergence. 

Now, fix $m\in\N$ such that $F_m$ has positive probability. If $(\omegac,\omegad)\in F_m$, we define the three disjoint branches $b^{(i)} = \tau^{(i)}|_{\R^d\setminus Q(0,n)}$, $1\le i\le 3$. Then, for all $\epsilon>0$ sufficiently small, the event
\begin{align*}
	G_m &=\{ (\omegac,\omegad) \in F_m \ | \ \forall i: \ \partial( b^{(i)}|_{A_{m,n}} ) = f_1 + f_2  \text{ such that $f_1$ and $f_2$ are cycles }\\
	&\qquad\qquad \text{ but not boundaries w.r.t.\ $\cK_r(\cP|_{ Q(0,m)\setminus Q(0,n)})$ and $f_1\subseteq A_{n+\mu(r)-\epsilon,n}$ } \} 
\end{align*}
has positive probability, too.

Now let $H_m$ consist of all $(\omegac,\omegad)\in \Omega^\circ\times \Omega^\dagger$ such that on the one hand, for this very $\omegad$, there is an $\wt\omega^\circ\in\Omega^\circ$ with $(\wt\omega^\circ,\omegad)\in G_m$ (for short, ``$\omegad$ enjoys the property $P_1$'') and on the other hand there are the following chains: A chain $r$, constructed from the elements of $\cPc_n(\omegac)$, and a chain $c$, constructed from the elements of $[\cPc_n(\omegac)\cup \cPd_n(\omegad)] \cap A_{m,n}$ with the property that the chain $\tau = r + c + \sum_{i=1}^3 b^{(i)}$ satisfies $\partial( \tau|_{Q(0,R)} ) \subseteq A_{R,R-\mu(r)}$ for all $R\ge m$ (for short, ``the pair $(\omegac,\omegad)$ enjoys the property $P_2$'').

Then $H_m$ can be formulated in an abstract way as
$$
	H_m = \{(\omegac,\omegad)\in \Omega^\circ\times \Omega^\dagger :  \text{ $\omegad$ has $P_1$ and $(\omegac,\omegad)$ has $P_2$}\}.
$$

We show that $H_m$ has positive probability. Since $G_m$ has positive probability, we have on the one hand
$$\p^{\dagger}(\{\omegad: \omegad \text{ has } P_1 \}) \ge \p^\circ\otimes\p^\dagger(G_m) > 0
$$
and we have on the other hand
\begin{align*}
	\p(H_m) &= \int_{\Omega^\dagger} \1{\omegad \text{ has } P_1 }     \p^{\circ}(\{ \omegac :  (\omegac,\omegad) \text{ has } P_2 \} ) \ \p^{\dagger}(\diff \omegad).
\end{align*}
Consequently, it remains to show $\p^{\circ}( \{ \omegac :  (\omegac,\omegad) \text{ has } P_2 \} ) > 0$ for almost all $\omegad$ which have $P_1$. 

To this end, we choose an $\omegad$ which has $P_1$. Conditionally on $\omegad$ and relying on classical results for triangulations, there is a finite set $P$ inside $Q(0,n)$ (with all elements in general position) such that $r$, $c$ and $\tau$ exist as laid out and all simplices, which are involved in the chain $r$, have a filtration time of at most $\mu(r)/2$. This entails that we can move the vertices of a specific $q$-simplex $\sigma$ in $\cK_r(P)$ by at most $\mu(r)/(4(q+1))$ and still have a filtration time of $\sigma$ of at most $3\mu(r)/4$.

Moreover, as $P\cup \cPd_n(\omegad)|_{Q(0,n+\mu(r))}$ is finite, there is a $\delta>0$ such that all vertices of $P$ can be moved by at least $\delta>0$ without adding (resp.\ removing) another simplex to (resp.\ in) the complex $\cK_r(P \cup \cPd_n(\omegad)|_{Q(0,n+\mu(r))} )$. Note that the probability that the elements of $\cPd_n|_{Q(0,n+\mu(r))}$ do not entail a sharp filtration time of exactly $\mu(r)$ is 1, which is what we implicitly assume in the choice of $\omegad$. (This effect of local constancy is also studied by Chazal and Divol \cite[Lemma 13]{chazal2018density} for filtration functions generated by (random) point clouds.)

Thus, for this specific $\delta>0$, the probability of the following event is positive: There is a realization of $\cPc_n$ which has $\# P$ elements and for each $p\in P$ there is an element of $\cPc_n$ in the $\delta$-neighborhood of $p$.

This proves $\p^{\circ}( \{ \omegac :  (\omegac,\omegad) \text{ has } P_2 \} ) > 0$ for almost all $\omegad$ which have $P_1$. Thus, $H_m$ occurs with positive probability and we arrive at $\p(E_m) \ge \p(H_m) > 0$. This completes the proof.
\end{proof}

\begin{proof}[Proof of Theorem~\ref{T:StrongStabAddOne}]
Let $q\in\{0,\ldots,d-1\}$ be fixed. For simplicity, we write $\cP=\cP(\lambda)$ and $S=S^{(r,s)}_q$. (Recall that $S$ is defined in (\ref{Def:StrongStabilizationRadius2}).) $S$ is $a.s.$ finite by Theorem~\ref{Thrm:StrongStabilization}.
By the definition of $\wt\rho^{\,r}_q$ and $\wt\rho^{\,s}_{q+1}$, the following two functionals do not change when changing the configuration outside $B(0,S)$:
\begin{align}\label{E:StrongStabAddOne1}
	A \mapsto \dim \frac{Z_q (\cK_r(\cP\cup\{0\}\cup A))}{Z_q (\cK_r(\cP\cup A))} \text{ and } A \mapsto \dim \frac{B_q (\cK_s(\cP\cup\{0\}\cup A))}{B_q (\cK_s(\cP\cup A))},
\end{align}
where $A\subseteq\R^d\setminus B(0,S)$ is finite. 

In order to conclude the case for the persistent Betti function $\beta^{r,s}_q$, we show that
\begin{align}\label{E:StrongStabAddOne2}
	A \mapsto \dim \frac{Z_q (\cK_r(\cP\cup\{0\}\cup A)) + B_q (\cK_s(\cP\cup\{0\}\cup A)) }{Z_q (\cK_r(\cP\cup A)) + B_q (\cK_s(\cP\cup A))},
\end{align}
$A\subseteq\R^d\setminus B(0,S)$ finite, is constant, too. This implies the case for the persistent Betti function by making use of the dimension formula  ($\dim (U+V) + \dim (U\cap V) = \dim U + \dim V$ for two finite-dimensional linear spaces $U,V$) in conjunction with (\ref{E:StrongStabAddOne1}).

On the one hand, assume that the map in (\ref{E:StrongStabAddOne2}) increases, when going from a set $A_1$ to another set $A_2$ outside $B(0,S)$. Then there is a $q$-chain $c^*$ which is not 0 modulo $Z_q (\cK_r(\cP\cup\{0\}\cup A_2))$ and not 0 modulo $B_q (\cK_s(\cP\cup\{0\}\cup A_2))$ but 0 modulo $Z_q (\cK_r(\cP\cup A_2))$ or 0 modulo $B_q (\cK_s(\cP\cup A_2))$. This contradicts the conclusion for the mappings in \eqref{E:StrongStabAddOne1}.

On the other hand, assume that the map decreases, when going from a set $A_1$ to $A_2$. This time there is a $q$-chain $c^*$ which is 0 modulo $Z_q (\cK_r(\cP\cup\{0\}\cup A_2))$ or 0 modulo $B_q (\cK_s(\cP\cup\{0\}\cup A_2))$ but not 0 modulo $Z_q (\cK_r(\cP\cup A_2))$ and not 0 modulo $B_q (\cK_s(\cP\cup A_2))$. This is again a contradiction.

Consequently, the persistent Betti number $\beta^{r,s}_q$ does not change with the configuration outside $B(0,S)$.
\end{proof}

\begin{proof}[Proof of Lemma~\ref{L:StrongAndWeakStabilization}]
Clearly, if the right-hand side is infinite, there is nothing to prove. So, assume that it is finite. Define $\cK_{s,a} = \cK_s( P\cap B(z_Q,a))$ and $\cK'_{s,a} = \cK_s((P\cup Q)\cap B(z_Q,a))$. We show
\[
			\dim Z_q( \cK'_{r,a}) - \dim Z_q( \cK_{r,a}) = \text{const. } \text{ and }  \dim B_q(\cK'_{s,a}) - \dim B_q(\cK_{s,a}) = \text{const. } 
\]
for all $a \ge  S_q^{(r,s)}(P,Q)$. By definition of $\wt\rho_q^{\,r}(P,Q)$, this is true for difference involving the cycle groups $Z_q( \cK'_{r,a})$ and $Z_q( \cK_{r,a})$ and $a\ge \wt\rho_q^{\,r}(P,Q)$.
Moreover, again by the definition of $S_q^{(r,s)}(P,Q)$, the difference $ \dim B_q \cK'_{s,a} - \dim B_q \cK_{s,a}$ is constant for all $a\ge S_q^{(r,s)}(P,Q)$. 

 Using the dimension formula $\dim (U+V) + \dim (U\cap V) = \dim U + \dim V$, it only remains to consider the difference $\dim ( Z_q (\cK'_{r,a}) + B_q (\cK'_{s,a}) ) - \dim( Z_q (\cK_{r,a}) + B_q (\cK_{s,a}) )$.

First, assume that this difference increases at an $a \ge S_q^{(r,s)}(P,Q)$. So, there is a $q$-chain $c^*$ such that $c^*\neq 0$ mod $B_q (\cK'_{s,a-})$ and $c^*\neq 0$ mod $Z_q (\cK'_{r,a-})$ but $c^*=0$ mod $B_q (\cK_{s,a-})$ or $c^*= 0$ mod $Z_q (\cK_{r,a-})$. This is a contradiction.
Second, if the difference decreases at an $a \ge S_q^{(r,s)}(P,Q)$, there is again a $q$-chain $c^*$ such that $c^*\neq 0$ mod $B_q (\cK_{s,a-})$ and $c^*\neq 0$ mod $Z_q (\cK_{r,a-})$ but $c^*=0$ mod $B_q (\cK'_{s,a-})$ or $c^*= 0$ mod $Z_q (\cK'_{r,a-})$. Again, this is a contradiction.
\end{proof}

\subsection{The asymptotic normality for the Poisson process}\label{S:Poisson}

Recall that $\cP$ and $\cP'$ denote independent Poisson processes with unit intensity on $\R^d$ and $B_n = [-2^{-1} n^{1/d},2^{-1} n^{1/d}]^d$. For $z \in \Z^d$, set $\cP''(z) = (\cP\setminus Q(z) )\cup (\cP'\cap Q(z))$, and let
$$\Delta^{r,s}_z (B_n) =  \beta^{r,s}_q(\cK( \cP \cap B_n)) - \beta^{r,s}_q(\cK( \cP''(z) \cap B_n)).$$
\begin{lemma}\label{L:asStabilizationPoisson}
For each $(r,s)\in\Delta$ and for each $z\in \Z^d$, there are random variables $\Delta^{r,s}_z(\infty)$ and $N_0=N_0(z,(r,s))\in\N$ such that $\Delta^{r,s}_z(B_n) \equiv \Delta^{r,s}_z(\infty)$ $a.s.$ for all $n\ge N_0$.
\end{lemma}
\begin{proof}
Let $z\in\Z^d$ and $(r,s)\in\Delta$ be arbitrary but fixed.
We utilize Lemma 3.1 in \cite{penrose2001central} to obtain a random variable $\Delta^{r,s}_z(\infty)$ such that
$$
	\lim_{n\to\infty} \Delta^{r,s}_z(B_n) = \Delta^{r,s}_z(\infty) \text{ with probability 1}.
$$
The existence of $N_0$ follows from the fact that Betti numbers and consequently the $\Delta^{r,s}_z(B_n)$ are integer valued.
\end{proof}

\begin{proposition}\label{P:CovarianceUnitPoisson}
For each two pairs $(r,s),(u,v)\in\Delta$,
\begin{align*}
	\gamma((u,v),(r,s)) &\coloneqq \E{ \E{\Delta^{r,s}_0(\infty) | \cF_0} \ \E{ \Delta^{u,v}_0(\infty) | \cF_0} } \\
	&=			\lim_{n \rightarrow \infty } n^{-1} \covar{ \beta^{u,v}_q(\cK( \cP \cap B_n)) }{ \beta^{r,s}_q(\cK( \cP \cap B_n))	} .
\end{align*}
\end{proposition}

\begin{proof}
Let $(r,s)$ and $(u,v)$ be arbitrary but fixed. Define the two functionals
$$
	h_1 \coloneqq \beta^{r,s}_q(\cK( \cdot )) \text{ and } h_2 \coloneqq \beta^{u,v}_q(\cK( \cdot )) 
$$
First, observe that
\begin{align}
\begin{split}\label{E:CovEq}
	&2 n^{-1} \covar{ h_1( \cP \cap B_n) }{ h_2( \cP \cap B_n)	} \\
	&= n^{-1} \ \V{ h_1( \cP \cap B_n) +  h_2( \cP \cap B_n) }  \\
	&\quad - n^{-1} \V{ h_1( \cP \cap B_n) } - n^{-1} \V{ h_2( \cP \cap B_n) }.
	\end{split}
\end{align}
In the following we prove the convergence of each of the three terms on the right-hand side to an appropriate limit as $n\to\infty$.

Using the established strong stabilization (Theorem~\ref{T:StrongStabAddOne}), the convergence of each term in question follows with Theorem 3.1 of \cite{penrose2001central} once the Poisson bounded moments condition is satisfied, i.e., for $i\in\{1,2\}$
$$
	\sup_{A \in \cB: 0\in A} \E{ | h_i( [\cP\cap A] \cup \{0\} ) - h_i(\cP\cap A) |^4 } < \infty, 
$$
where $\cB = \{ B_n + x: x\in \R^d, n\ge 1\}$. Validity of the Poisson bounded moment condition follows immediately from the Geometric Lemma. We skip the details and refer to \cite{yogeshwaran2017random} (proof of Lemma 4.1) instead.

Thus, using Theorem 3.1 of \cite{penrose2001central},
\begin{align*}
	\lim_{n\to\infty} n^{-1} \V{ h_1( \cP \cap B_n) } = \E{ \E{ \Delta^{r,s}_0(\infty) | \cF_0 }^2 }, \\
	\lim_{n\to\infty} n^{-1} \V{ h_2( \cP \cap B_n) } = \E{ \E{ \Delta^{u,v}_0(\infty) | \cF_0 }^2 }.
\end{align*}
Moreover, using the additivity of the add one cost, the functional $h_1 + h_2$ enjoys the strong stabilization just as $h_1$ and $h_2$. Hence, applying Theorem~3.1 of \cite{penrose2001central} once more, yields
\begin{align*}
	\lim_{n\to\infty} n^{-1} \V{ h_1( \cP \cap B_n) + h_2( \cP \cap B_n) } = \E{ \E{ \Delta^{r,s}_0(\infty) + \Delta^{u,v}_0(\infty) | \cF_0 }^2 }.
\end{align*}
Consequently, relying on \eqref{E:CovEq}
$$
 \lim_{n\to\infty} 2 n^{-1} \covar{ h_1( \cP \cap B_n) }{ h_2( \cP \cap B_n)	} = 2 \E{ \E{ \Delta^{r,s}_0(\infty) | \cF_0 } \E{\Delta^{u,v}_0(\infty) | \cF_0 } }
$$
and the claim follows.
\end{proof}

\begin{proof}[Proof of Theorem~\ref{T:GaussianLimitPoisson}]
We show the multivariate asymptotic normality by considering finite linear combinations of persistent Betti numbers. For $a_1,\ldots,a_\ell\in\R$ and $n\in\N,$ let $H( n^{1/d} \cP_n) = \sum_{i=1}^\ell a_i \beta^{r_i,s_i}_q(\cK( n^{1/d} \cP_n))$. We will verify that
\begin{align}\begin{split}\label{E:GaussianLimitPoisson1}
	&n^{-1} \ \V{H( n^{1/d} \cP_n)} \to \sigma^2 \\
	&\text{ and } n^{-1/2} \big( H( n^{1/d} \cP_n)  - \mathbb{E}[H( n^{1/d} \cP_n)] \big) \Rightarrow \cN(0,\sigma^2) \text{ as } n\to\infty,
\end{split}
\end{align}
where $\sigma^2 = \int_{[0,1]^d} \ol \sigma^2( \kappa(x)) \ \intd{x}$; here $\ol \sigma^2(\lambda)$ is the corresponding limiting variance if the Poisson point process $n^{1/d} \cP_n$ in the functional $H$ is replaced by a homogeneous Poisson process $\cP(\lambda)$ (with intensity $\lambda$) restricted to $B_n$. Now, we have
\begin{align*}
	\ol \sigma^2(\lambda) &= \lim_{n\to\infty} n^{-1} \V{H(\cP(\lambda)|_{B_n})} \\
	&=  \sum_{i,j=1}^\ell a_i a_j \ \lim_{n\to\infty}  n^{-1} {\rm Cov}[ \beta^{r_i,s_i}_q( \cK( \cP(\lambda)|_{B_n}) ),   \beta^{r_j,s_j}_q( \cK( \cP(\lambda)|_{B_n}) ) ] \\
	&= \lambda \sum_{i,j=1}^\ell a_i a_j  \lim_{n\to\infty}  (\lambda n)^{-1} {\rm Cov}[ \beta^{\lambda^{1/d} r_i,\lambda^{1/d} s_i}_q( \cK( \cP(1)|_{\lambda^{1/d}B_n}) ), \\
	&\hspace*{7cm} \beta^{\lambda^{1/d} r_j, \lambda^{1/d}s_j}_q( \cK( \cP(1)|_{\lambda^{1/d}B_n}) ) ] \\
	&= \lambda \sum_{i,j=1}^\ell a_i a_j  \gamma( \lambda^{1/d} (r_i, s_i), \lambda^{1/d} (r_j, s_j) ),
\end{align*}
where the last equality follows from from Proposition~\ref{P:CovarianceUnitPoisson}, and where $\gamma$ is the limiting covariance function of an underlying homogeneous Poisson process with unit intensity on $\R^d$. Hence,
\begin{align}\begin{split}\label{E:GaussianLimitPoisson2}
	\sigma^2 &= \sum_{i,j=1}^\ell a_i a_j \ \int_{[0,1]^d} \gamma( \kappa(x)^{1/d} (r_i,s_i),\kappa(x)^{1/d} (r_j,s_j) ) \ \kappa(x) \ \intd{x}.
\end{split}\end{align}
Consequently, once the statements in \eqref{E:GaussianLimitPoisson1} are verified, the proof is complete.

Using Theorem 3.3 of \cite{trinh2019central}, the statements in \eqref{E:GaussianLimitPoisson1} hold if, apart from the strong stabilization (Theorem~\ref{T:StrongStabAddOne}), the following two moment conditions are satisfied for cubes of the type $W = z + [0,a)^d \subseteq \R^d$ with $z\in\R^d,a\in\R_+$, and for each pair $(r,s)$ with $r\le s$:
\begin{itemize}
	\item [(1)] the Poisson bounded moments condition: for some $p>2$
	\[
		\sup_n \sup_{y\in\R^d} \sup_{y\in W: 	\text{ cube}} \E{ | \fD_0 \beta^{r,s}_q( \cK( n^{1/d} \cP_n |_W - y ) ) |^p } <\infty.
	\]
	\item [(2)] the locally bounded moments condition: for each cube $W$ as above, there is a $p>2$ such that
	\[
		\sup_n \sup_{y\in\R^d} \E{ | \beta^{r,s}_q( \cK( n^{1/d} \cP_n |_{y+W} ) ) |^{p} } <\infty.
	\]
\end{itemize}
Both conditions (1) and (2) are immediate consequences of the Geometric Lemma (Lemma~\ref{L:GeometricLemma}). Indeed, to see (1), observe that the Geometric Lemma gives
\begin{align*}
	| \fD_0 \beta^{r,s}_q( \cK( n^{1/d} \cP_n |_W - y ) ) | &\le \sum_{j=q}^{q+1} \cK_j(  (  n^{1/d} \cP_n |_W ) \cup \{y\} , s) \setminus \cK_j(  (  n^{1/d} \cP_n |_W ) ,s ) \\
	&\le 2 |  n^{1/d} \cP_n  \cap B(y,\mu(s) ) |^{q+1}.
\end{align*}
This last upper bound is stochastically dominated by $| \cP( \| \kappa \|_\infty )\cap B(0,\mu(s)) |^{q+1}$, which does not depend on $n,y$ or $W$, and the moment $\E{ | \cP( \|\kappa\|_\infty )\cap B(0,\mu(s)) |^{p(q+1)}}$ is finite for each $p,q\ge 0$. 

Regarding (2), we have again by Lemma~\ref{L:GeometricLemma} and the translation invariance that $\beta^{r,s}_q( \cK(n^{1/d} \cP_n |_{y+W} ))$ is stochastically dominated by $2 |\cP(\| \kappa \|_\infty)|_W|^{q+2}$, which does not depend on $y$ and $n$. Moreover, for each cube $W$ the moment $\E{ |\cP(\| \kappa\|_\infty)|_W|^{p(q+2)} }$ is finite for all $p,q\ge 0$. This completes the proof.
\end{proof}

\subsection{The asymptotic normality for the binomial process}\label{S:Binomial}	

For each $n\in\N$, let $(U_{m,n}\colon m\in\N)$ be a sequence of binomial processes such that $U_{m,n} = ( Y_{1,n},\ldots,Y_{m,n} )$ for i.i.d.\ sequences $(Y_{i,n}\colon i\in\N)$ with common marginal density $\kappa$.
\begin{proof}[Proof of Theorem~\ref{T:GaussianLimitBinomial}]
Let $\ell\in\N$, $a_1,\ldots,a_\ell\in\R$. We apply the abstract Theorem 3.9 in Trinh \cite{trinh2019central} to the functional $H( n^{1/d} \mX_n) = \sum_{i=1}^\ell a_i \beta^{r_i,s_i}_q(\cK( n^{1/d} \mX_n))$. 
Since we established the Poisson bounded moments condition and the locally bounded moments condition in the proof of Theorem~\ref{T:GaussianLimitPoisson} and since the $q$th persistent Betti number obtained from a finite set is polynomially bounded in the cardinality of this set, it is enough to verify for the add one cost function
\begin{align*}
				&\sup_{ n\in\N} \sup_{ m\in [n(1-\eta),n(1+\eta)]}  \E{ | \beta^{r,s}_q(\cK( n^{1/d}U_{m+1,n} )) - \beta^{r,s}_q(\cK( n^{1/d}U_{m,n} ))|^4} < \infty 
\end{align*}
for a some $\eta>0$. This can be verified with the geometric lemma (Lemma~\ref{L:GeometricLemma}) similar to Lemma~4.1 in Yogeshwaran et al.\ \cite{yogeshwaran2017random}, we omit the details.

Hence, the conditions of Theorem 3.9 in Trinh \cite{trinh2019central} are satisfied and
\begin{align*}
	&\frac{ \V{ H(n^{1/d} \mX_n ) } }{n} \to \tau^2,  \qquad  \frac{H(n^{1/d} \mX_n) - \E{H(n^{1/d} \mX_n)}}{\sqrt{n}} \Rightarrow N(0,\tau^2),
\end{align*}
where $\tau^2$ is given by the following relation: 
$
	\tau^2 = \sigma^2 - \Big(\int_{[0,1]^d} \E{ \ol\Delta( \kappa(x))} \kappa(x) \ \intd{x} \Big)^2
$
for $\sigma^2$ given in \eqref{E:GaussianLimitPoisson2} and 
\begin{align*}
	\E{ \ol\Delta( \lambda ) } &= \sum_{i=1}^\ell a_i \ \E{ \lim_{n\to\infty} \fD_0 \beta^{r_i,s_i}_q( \cK( \cP(\lambda)|_{B_n} )) } \\
	&=  \sum_{i=1}^\ell a_i \  \E{ \fD_0 \beta^{r_i,s_i}_q( \cK( \cP(\lambda)\cap B(0,S^{(r_i,s_i)}_q(\lambda) ))) } \\
	&= \sum_{i=1}^\ell a_i \  \E{ \fD_0 \beta^{\lambda^{1/d}r_i, \lambda^{1/d} s_i}_q( \cK( \cP(1)\cap B(0,S^{(\lambda^{1/d} r_i,\lambda^{1/d} s_i)}_q(1) ))) } \\
	&= \sum_{i=1}^\ell a_i \ \alpha(\lambda^{1/d} (r_i,s_i) ).
	\end{align*}
Thus, we have for a random variable $X$ distributed with density $\kappa$
\begin{align*}
	\tau^2 &= \sum_{i,j=1}^\ell a_i a_j \ \bigg\{ \E{ \gamma( \kappa(X)^{1/d} (r_i,s_i), \kappa(X)^{1/d} (r_j,s_j) ) } \\
	&\qquad\qquad\qquad\quad - \E{ \alpha(\kappa(X) ^{1/d} (r_i,s_i) ) } \E{ \alpha(\kappa(X) ^{1/d} (r_j,s_j) ) } \bigg\}.
\end{align*}
This completes the proof.
\end{proof}

\acks
The authors are very grateful to an AE and two referees for their careful reading and their suggestions which greatly improved the manuscript. Moreover, the authors thank Khanh Duy Trinh for his interest in this paper and for pointing out a mistake in a preliminary version.
Johannes Krebs was supported by the German Research Foundation (DFG), grant numbers KR-4977/1-1 and KR-4977/2-1. Wolfgang Polonik was partially supported by the National Science Foundation (NSF), grant number DMS-2015575.


\appendix 

\section{Appendix}\label{AppendixTightnessOfStabilization}
First we give the proof of Theorem~\ref{Thrm:StrongStabilizationAppl}
\begin{proof}[Proof of Theorem~\ref{Thrm:StrongStabilizationAppl}]
Recall that $\mu(r)$ is an upper bound on the diameter of a simplex with filtration time at most $r \ge 0$. Note that the statements for the radius of strong stabilization $\wt\rho_q^{\,r}$, together with the relation 
\[
	\rho^{(r,s)}(P,Q) \le S^{(r,s)}_q (P,Q) \qquad\text{for each }(r,s)\in \Delta, q\in\{0,\ldots,d-1\}
	\]
from Lemma~\ref{L:StrongAndWeakStabilization}, allows us to conclude the results for the radius of weak stabilization. So it remains to prove the statement for the strong stabilization property. We proceed for each $q$ separately; clearly this is no restriction.

In the following, if $Q_1$, $Q_2$, $r$ and $q$ are fixed, we just write $\wt\rho(\kappa)$ for $	\wt\rho_q^{\,r}(\cP(\kappa)\cup Q_1,Q_2) $ if the Poisson process $\cP(\kappa)$ has intensity $\kappa$.

In the remainder of the proof, we assume w.l.o.g. that the Poisson process $\cP(\kappa)$ on $\R^d$ is coupled to a homogeneous Poisson process $\cP$ on $\R^{d+1}$ with intensity 1 via a space-time coupling as follows
\begin{align}\label{DefCouplingPoisson}
		\cP(\kappa) = \{x|\, \exists\; t: 0 \le t \le \kappa(x), (x,t) \in \cP\}.
\end{align}
\textit{Proof of (1):} Let $r,\epsilon>0$ and $q\in\{0,\ldots, d-1\}$ be arbitrary but fixed. We first consider the case of adding exactly one additional point to the Poisson process, so $Q_1 = \emptyset$ and $Q_2=\{0\}$. We show that there is an $L>0$ such that $\p( 	\wt\rho(\lambda) > L)\le 2\epsilon$ for each $\lambda\in\R_+$. The generalization then works along the same lines.

We rely on the following ideas: If the intensity $\lambda$ is sufficiently small, there are no points of the Poisson process in the neighborhood of the additional point 0 with high probability; hence, adding 0 does not create a new cycle. 
Moreover, if the intensity $\lambda$ is sufficiently large, then an annulus around the origin is covered by the points of the Poisson process with high probability; hence the impact of adding 0 is only local.
Finally, we rely on the space-time coupling for the intermediate intensities.

First, there is a $\ul\kappa\in\R_+$ such that
\[
			\p( |\cP(\lambda) \cap B(0,\mu(r)) | = 0 ) \ge 1-\epsilon \text{ for all } \lambda\le \ul\kappa. 
\]
This means that with high probability and for all $q \ge 1$, we have, for all $\lambda$ below this threshold, that including $\{0\}$ does not create any additional $q$-simplex. Hence, if $\lambda\le \ul\kappa$, then $\p( 	\wt\rho(\lambda) > L)\le \epsilon$ for $L\ge \mu(r)$.

Also, there is an intensity $\ol\kappa$ such that with high probability all changes in the add one cost function which are caused by including the origin are limited to a deterministic neighborhood; we carry out the argument simultaneously for the \v Cech and Vietoris-Rips complex. Indeed, let $\delta \le \mu(r)(1-1/\sqrt{2})/\sqrt{d}$ sufficiently small such that $4\mu(r) /\delta \in\N$ and consider a partition of $A:=A_{2\mu(r)+\delta,2\mu(r)} =  [- 2\mu(r)-\delta, 2\mu(r)+\delta]^d \setminus (- 2\mu(r), 2\mu(r))^d$ with subcubes $(C_i)_{i\in I}$ of edge length $\delta$. We study $\cU = \bigcap_{i\in I} \{\#(\cP(\lambda) \cap C_i) \ge d\}$ (\textit{``each subcube contains at least $d$ Poisson points''}). Then, there is a $\ol\kappa\in\R_+$ depending on $r$ and $\delta$ such that
\[
		\p(\cU) = \prod_{i \in I} \p(\#(\cP(\lambda) \cap C_i) \ge d) \ge 1-\epsilon \text{ for all } \lambda\ge \ol\kappa.
\]
Define $L_0 \coloneqq (2\mu(r)+\delta)\sqrt{d}$. We show $\{  \wt\rho(\lambda) \le L_0 \} \supseteq \cU$  which in turn implies $\p( \wt\rho(\lambda) > L) \le \epsilon$ for $L \ge L_0$ if $\lambda \ge \ol{\kappa}$.

So, assume $\cU$. Let $\sigma^r_{q,1},\ldots,\sigma^r_{q,m_q}$ be the $q$-simplices which contain the origin. Clearly, these simplices are all contained in $[- \mu(r), \mu(r)]^d$. In particular, they do not intersect with $A$, which itself is homeomorphic to a $(d-1)$-cycle.

Plainly, we can triangulate $A$ with $(d-1)$-simplices with filtration time at most $ 2\delta \sqrt{d} < \mu(r)$: Indeed, consider two adjacent cubes $C_1,C_2$ each containing a $(d-1)$-simplex, $\sigma_1=\{x_0,\ldots,x_{d-1}\},\sigma_2=\{y_0,\ldots,y_{d-1}\}$, say. (Here adjacent means that the closures of the cubes have a nonempty intersection.) Then we can connect $\sigma_1$ and $\sigma_2$ with the $(d-1)$-simplices $\{x_0,\ldots,x_{i-1},y_{i},\ldots,y_{d-1}\}$ ($1\le i\le d-1$) and each of these simplices has a filtration time of at most $2\delta\sqrt{d}$.

Moreover, let $\sigma=\{x_0,\ldots,x_{k-1}\}$ be a $(k-1)$-simplex in $A_{2\mu(r) + \mu(r)/2, 2\mu(r)-\mu(r)/2}$ for a generic $k\in\N$. Then there is a $(k-1)$-simplex $\sigma^*=\{y_0,\ldots,y_{k-1}\}$ in $A$ such that the Euclidean distance between any two elements of $\sigma$ and $\sigma^*$ is at most $\sqrt{ (\mu(r)/2)^2 + (\mu(r)/2)^2 } + \sqrt{d} \delta$, which in turn is at most $\mu(r)$. Consequently, all $k$-simplices of the type $\{x_0,\ldots,x_i,y_i,\ldots,y_{k-1}\}$ have a filtration time at most $\mu(r)$.

This shows that conditional on $\cU$ any cycle in $A_{2\mu(r) + \mu(r)/2, 2\mu(r)-\mu(r)/2}$ is equivalent to a cycle in $A_{2\mu(r)+\delta,2\mu(r)}$ and, thus, ia a boundary as well. Hence, $\wt\rho(\lambda) \le L_0$ in this case.
In particular, we have $\p( \wt\rho(\lambda)> L)\le\epsilon$ for all $\lambda \ge \ol\kappa$ and all $L\ge L_0 = \sqrt{d}(2r+\delta)$.
 
It remains to check intensities $\lambda\in[\ul\kappa,\ol\kappa]$. Assume there is an $\epsilon>0$ such that
\[
			\limsup_{L\rightarrow \infty} \sup_{\lambda\in[\ul\kappa,\ol\kappa]} \p( \wt\rho(\lambda) > L ) > 2\epsilon.
\]
Then we can find sequences $(L_n)_n$ and $(\lambda_n)_n$ such that $L_n\rightarrow\infty$ and $\lambda_n\rightarrow\lambda^*\in[\ul\kappa,\ol\kappa]$ with the property that $\p( \wt\rho(\lambda_n)>L_n) > \epsilon$ for all $n\in\N$. However, there is an $L^*\in\R_+$ such that $\p(\wt\rho(\lambda^*)> L^*) < \epsilon/4$ as $\wt\rho$ is $a.s.$ finite by Theorem~\ref{Thrm:StrongStabilization}. Also due to the coupling of the Poisson processes (\ref{DefCouplingPoisson}), and because $\cP(\lambda)$ is a simple point process there are random $\ol{\delta}, \ul{\delta} > 0$ (depending on the choice of $L^*$) such that $\cP(\lambda)$ does not contain any points in $B(0,L^* + \mu(r))\times [\lambda^* - \ul\delta, \lambda^* + \ol\delta]$, viz.,
\[
	\ul\delta = \big(\lambda^* - \inf\{ \lambda<\lambda^*: \cP(\lambda)( B(0,L^* + \mu(r))\times [\lambda,\lambda^*] ) = 0 \}\big)/2
\]
and $\ol\delta = \big( \sup\{ \lambda>\lambda^*: \cP(\lambda)( B(0,L^* + \mu(r))\times [\lambda^*,\lambda] ) = 0 \} - \lambda^* \big)/2$. This means, for all $\lambda\in [\lambda^* -\ul\delta,\lambda^* +\ol\delta]$,
\[
			\cP(\ul\lambda)|_{B(0,L^* + \mu(r)) } \equiv \cP(\lambda)|_{B(0,L^* + \mu(r)) } \equiv \cP(\ol\lambda)|_{B(0,L^* + \mu(r)) }.
\]
Note that $\{ \lambda^*-\ul\lambda \ge \delta' \} \cap \{ \ol\lambda - \lambda^* \ge \delta' \} \supseteq \{ \cP( B(0,L^* + \mu(r))\times [\lambda^*-2\delta',\lambda^*+2\delta'] ) = 0 \}$ for $\delta'>0$. Consequently, there is an $\delta'\in\R_+$ such that the event $\{ \lambda^*-\ul\lambda>\delta'\} \cap \{ \ol\lambda - \lambda^* > \delta' \}$ has probability of at least $1-\epsilon/4$. Consequently, for all $n$ large enough such that $L_n\ge L^*$ and $|\lambda_n - \lambda^*|\le \delta'$

\begin{align*}
		\p( \wt\rho(\lambda_n)>L_n) &\le \p\big( \wt\rho(\lambda_n)>L_n, \wt\rho(\lambda^*) \le L^*, \lambda^*-\ul\lambda>\delta', \ol\lambda - \lambda^* > \delta' \big)  \\
		&\quad + \p\big(\wt\rho(\lambda^*) > L^* \big) + \p\big( \{\lambda^*-\ul\lambda\le \delta'\} \cup \{\ol\lambda - \lambda^*\le \delta'\} \big) \le \epsilon/2
\end{align*}
because $ \{ \wt\rho(\lambda_n)>L_n, \wt\rho(\lambda^*) \le L^*, \lambda^*-\ul\lambda>\delta', \ol\lambda - \lambda^* > \delta' \} = \emptyset$. This contradicts $\p( \wt\rho(\lambda_n)>L_n) > \epsilon$ for all $n\in\N$. Thus, the laws of $\{\wt\rho(\cP(\lambda),\{0\})):\lambda\in\R_+\}$ are tight. 

A similar reasoning shows that also the laws of $\{\wt\rho(\cP(\lambda)\cup Q_1,Q_2), \lambda \in \R_+, Q_1,Q_2\in \fQ_m\}$ are tight. At this point, it is essential that we have a uniform upper bound on the parameter $a^*(r)$ as follows
\begin{align}\label{E:ParaA}
	a^*(r) = L_{Q_2} + \mu(r) \le \sqrt{d} + \mu(\ol r)
\end{align}
for each $r \le \ol r$ and each $Q_2\subseteq Q(0)$. So, instead of computing the radius of strong stabilization by taking the infimum in \eqref{Def:StrongStabilizationRadius1} over all $\{ R: R \ge a^*(r)\}$, we can first take the infimum over the smaller set $\{R: R \ge \sqrt{d} + \mu(\ol r) \}$, which does not depend on $Q_2$ and $r$, in order to obtain a modified radius of strong stabilization, which is not smaller than the original one in \eqref{Def:StrongStabilizationRadius1}. Verifying the claim for this modified radius, implies then the claim for the original radius and for the rest of the proof of (1), we argue with this modified radius, which we denote as well by $\wt\rho_q^{\,r}$ abusing the notation slightly.

We sketch the remaining steps: Using the same techniques, we easily see that there are upper and lower bounds $\ul\kappa$ and $\ol\kappa$ such that intensities $\lambda\notin[\ul\kappa,\ol\kappa]$ only have a local effect in the same sense as in the special case for $\{0\}$, i.e., for each $\epsilon>0 $ there are $\ul\kappa,\ol\kappa$ and an $L>0$ such that
\begin{align*}
				\sup_{Q_1,Q_2\in \fQ_m} \quad \sup_{\lambda\notin [\ul\kappa,\ol\kappa] } \p( \wt\rho(\cP(\lambda)\cup Q_1,Q_2) > L ) \le \epsilon.
\end{align*}
Also for intensities $\lambda\in [\ul\kappa,\ol\kappa]$, we can repeat the argument as in the special case treated above. Indeed, assume the contrary, namely,
\begin{align*}
		\limsup_{L\rightarrow \infty} \quad \sup_{Q_1,Q_2\in\fQ_m} \quad \sup_{\lambda\in[\ul\kappa,\ol\kappa]} \p( \wt\rho(\cP(\lambda)\cup Q_1,Q_2) > L ) > 2\epsilon,
\end{align*}
for some $\epsilon>0$. Then there are sequences $(Q_{n,1})_n, (Q_{n,2})_n, (\lambda_n)_n, (L_n)_n$ with the properties $Q_{n,1}\rightarrow Q^*_1, Q_{n,2}\rightarrow Q^*_2$ (considered as vectors the entries of which are elements in $[-2^{-1},2^{-1}]^d$) for two admissible elements $Q^*_1,Q^*_2 \in \fQ_m$ as well as $\lambda_n\rightarrow \lambda^*\in [\ul\kappa,\ol\kappa]$ and $L_n\rightarrow\infty$. Also these sequences satisfy
\[
			\p( \wt\rho(\cP(\lambda_n)\cup Q_{n,1},Q_{n,2}) > L_n ) > \epsilon\qquad \text{for all $n$}.
\]
Now, we can argue as before in the special case to obtain a contradiction. Hence, we arrive at the following result: For all $\epsilon>0$, for all $m\in\N$, for all $r\in \R_+$, there is an $L>0$ such that
\[
		\max_{q\in\{0,\ldots,d-1\}} \quad \sup_{\lambda\in\R_+} \quad \sup_{Q_1,Q_2\in\fQ_m} \p( \wt\rho_q^{\,r}(\cP(\lambda)\cup Q_1,Q_2) > L ) \le \epsilon.
\]

So far, we have been considering a fixed $r\le \ol r$. We now prove the general statement given in (1). To this end, we rely once more on \eqref{E:ParaA} and the induced modification of $\wt\rho_q^{\,r}$. Then, for $r\le \ol r$ and $\alpha\in\R_+$ arbitrary but fixed,
\begin{align}\begin{split}\label{E:ScalingEquality}
			\p( \wt\rho_q^{\,\alpha r}( \cP(\lambda)\cup Q_1, Q_2) > L ) &=  \p( \wt\rho_q^{\,r}( \alpha^{-1} \cP( \lambda)\cup \alpha^{-1}Q_1, \alpha^{-1}Q_2) > \alpha^{-1} L ) \\
			&= \p( \wt\rho_q^{\,r}( \cP(\alpha^d \lambda)\cup \alpha^{-1}Q_1, \alpha^{-1}Q_2) > \alpha^{-1} L ),
\end{split}\end{align}
using the scale invariance for the first equation and $\cL( \alpha^{-1} \cP(\lambda)) = \cL( \cP(\alpha^d \lambda))$ for the second equation.
Consequently, for all $\ol\alpha,\ul\alpha,\epsilon\in\R_+$ arbitrary but fixed, $\ul\alpha\le \ol\alpha$, there is an $L\in\R_+$ such that
\begin{align}\label{E:TightRhoUnif1}
			\sup_{\alpha\in [\ul\alpha,\ol\alpha]} \quad \sup_{\lambda\in\R_+} \quad \sup_{Q_1,Q_2\in\fQ_m} \p( \wt\rho_q^{\,\alpha r}( \cP(\lambda)\cup Q_1, Q_2) > L ) \le \epsilon.
\end{align}
This completes the considerations of part (1).

\textit{Proof of (2):} We use a suitable space-time coupling of Poisson processes. Let $r\in\R_+$ and $q\in\{0,\ldots,d-1\}$ be arbitrary but fixed. Write $\wt\rho$ for $\wt\rho_q^{\,r}$. Note that the law of $n^{1/d}\cP(n \nu)$ equals the law of $\cP( \nu(\cdot/n^{1/d}))$. 

First, we show that for all $\epsilon>0$, there is a $b>0$ and an $L>0$ such that 
\[
	\sup_{n\in\N} \sup_{z \in  B''_{n,L}} \p( \wt\rho( \cP( \nu(\cdot/n^{1/d})\cup Q_1, Q_2 ) > L) \le \epsilon
\]
 for all densities $\nu$ satisfying $\|\nu - \kappa\|_\infty \le b,$ and for all $Q_1,Q_2\in z+\fQ_m$, where $m\in\N$ is arbitrary but fixed.
 
From part (1) of the theorem, for each $\epsilon>0$ there is an $L>0$ such that, for the homogeneous Poisson process $\cP( \nu(z/n^{1/d})),$ it is true that
\[
			\p( \wt\rho( \cP( \nu(z/n^{1/d}))\cup Q_1, Q_2 ) > L ) \le \epsilon
\]
uniformly in $z\in [0,n^{1/d}]^d$, $Q_1,Q_2\in z+\fQ_m$ and $n\in\N$ and for all densities $\nu$. Moreover, with $K > 0,$ define the set
\begin{align*}
			A_{\nu,n}(K,z) = \Big\{x\in B(z,K) \ \Big| & \ \exists t\in \Big[\nu\Big( \frac{x}{n^{1/d}}\Big)\wedge \nu\Big( \frac{z}{n^{1/d}}\Big), \nu\Big( \frac{x}{n^{1/d}}\Big)\vee \nu\Big( \frac{z}{n^{1/d}}\Big) \Big] \\
			&\qquad \text{ and } (x,t)\in\cP	 \Big\}.
\end{align*}
By assumption, the function $\kappa$ is uniformly continuous with a certain modulus of continuity $\omega(\cdot)$. Hence, $|\kappa(x_1) - \kappa(x_2)|\le \omega(\delta)$, whenever $\|x_1-x_2\| \le \delta$ and $\omega(\delta)\to 0$ as $\delta\to 0$. Let $\nu$ be a density satisfying $\| \nu - \kappa\|_\infty \le b$. Then
\begin{align}
			&\p( A_{\nu,n}(K,z) \neq \emptyset ) \nonumber \\
			&\le \p\Big( \exists x\in B(z,K) \ \Big| \  \exists t \in \Big[\kappa\Big( \frac{x}{n^{1/d}}\Big)\wedge \kappa\Big( \frac{z}{n^{1/d}}\Big) - b, \kappa\Big( \frac{x}{n^{1/d}}\Big)\vee \kappa\Big( \frac{z}{n^{1/d}}\Big) + b \Big] \nonumber\\
			&\qquad \qquad \qquad \qquad\qquad \text{ and } (x,t)\in\cP \Big) \nonumber \\
			&\le \p\Big( \exists x\in B(z,K):\;\exists t \in \Big[0, \omega(K/n^{1/d}) + 2b\Big] \text{ and } (x,t)\in\cP \Big), \label{E:StrongStabilizationRadius1ab}
\end{align}
where the last inequality uses the stationarity of $\cP$. Clearly, given a value for $K$, there are $b>0$ and $n_0\in\N$ such that \eqref{E:StrongStabilizationRadius1ab} is small uniformly in $z\in B''_{n,K}$, $\nu$ in a $b$-neighborhood of $\kappa,$ and $n\ge n_0$.

We come to the conclusion. Let $\epsilon>0$ be arbitrary but fixed. We apply the result from part (1) and choose $L^*\in\R_+$ such that	$\p( \wt\rho( \cP(\lambda)\cup Q_1, Q_2) > L^* ) \le \epsilon/2$ is satisfied uniformly in $\lambda\in\R_+$, $Q_1,Q_2\in z+\fQ_m$ and $z\in\R^d$. 

Next, let $K^* = L^* + 2\mu(r)$. Choose $b>0$ and $n_0\in\N$ such that $\p( A_{\nu,n}(K^*,z) \neq \emptyset ) \le \epsilon/2$ for all $n\ge n_0$, for all $z\in  B''_{n,K^*}$ and for all $\nu$ such that $\|\nu - \kappa \|_\infty \le b$. Since by assumption $z\in  B''_{n,K^*}$, this implies 
\begin{align}\begin{split}\label{E:StrongStabilizationRadius1b}
			&\Big\{ \wt\rho( \cP( \nu( \cdot / n^{1/d} ) )\cup Q_1, Q_2) > K^*, \\
			&\qquad \wt\rho( \cP(\nu(z/n^{1/d}))\cup Q_1, Q_2) \le L^*, A_{\nu,n}(K^*,z) = \emptyset \Big\} = \emptyset.
\end{split}\end{align}
Consequently,
\begin{align}\begin{split}\label{E:StrongStabilizationRadius2}
			&\p( \wt\rho( \cP( \nu( \cdot / n^{1/d} ) ) \cup Q_1, Q_2) > K^* )\\
			&\le \p\big( \wt\rho( \cP( \nu( \cdot / n^{1/d} ) ) \cup Q_1, Q_2) > K^*,  \\
			&\qquad \wt\rho( \cP(\nu(z/n^{1/d}))\cup Q_1, Q_2) \le L^*, A_{\nu,n}(K^*,z) = \emptyset \big) + \epsilon = \epsilon,
\end{split}\end{align}
for all $z\in  B''_{n,K^*}$, for $n\ge n_0$ and for all $\nu$ such that $\sup |\nu - \kappa| \le b$.

The generalization to an entire parameter range for the filtration parameter follows now from the result in \eqref{E:TightRhoUnif1} (note that $\ol\lambda = \sup \kappa$ is an admissible choice in this equation) and by using a similar ansatz as in the derivation of \eqref{E:StrongStabilizationRadius2}: So, for each $\epsilon>0$, $\ol\alpha\ge \ul\alpha>0$, there are $L\in \R_+$, $b>0$ and $n_0\in\N$ such that for each $0\le q \le d-1$
\begin{align}\label{E:StrongStabilizationRadius3}
		\sup_{\ul\alpha\le \alpha \le \ol\alpha} \quad \sup_{n \ge n_0} \quad \sup_{Q_1,Q_2\in\fQ_{m}}	\p( \wt\rho_q^{\,\alpha r} ( \cP( \nu( \cdot / n^{1/d} ) ) \cup Q_1, Q_2) > L )  \le \epsilon,
\end{align}
for all densities $\nu$ in a $b$-neighborhood of $\kappa$ w.r.t.\ the sup-norm.
This yields then the first result given in part (2). The second result is now immediate: There is an $m\in\N$ such that with high probability, the number of Poisson points inside $Q(z)$ is at most $m$. So that we can then apply the previous result.

\textit{The proof of (3):} This time we couple the binomial process to a suitable Poisson process. Again, let $0\le q\le d-1$ and $r$ be arbitrary but fixed. We first study the radius $\wt\rho_q^{\,r}(n^{1/d}\mX_m,n^{1/d}X')$. First, note that for all $\epsilon>0$ and for all $L>0$ there is an $n_0\in\N$ such that $\p(n^{1/d} X'\notin B''_{n,L}) \le \ol\kappa n^{(d-1)/d} L n^{-1} \le \epsilon$ for all $n\ge n_0$. So, due to independence, for each $s>0,$
\[
			\p( \wt\rho_q^{\,s}( n^{1/d} \mX_m, n^{1/d} X') >L ) \le \sup_{z\in B''_{n,L}} \p\big( \wt\rho_q^{\,s}( n^{1/d} \mX_m, \{z\} ) >L \big) + \epsilon, \quad \forall n\ge n_0.
\]

For each $n\in\N$, let $V_{1,n},V_{2,n},\ldots$ be i.i.d.\ with density $\kappa(\cdot/n^{1/d})$. Let $\cP( \kappa(\cdot/n^{1/d})) = \{Z_{1,n},\ldots, Z_{N_n,n}\}$ be the Poisson process from \eqref{DefCouplingPoisson} for the intensity function $\kappa(\cdot/n^{1/d})$ and $N_n\sim \poi(n)$. Then $n^{1/d}\mX_m$ has the same distribution as the process
\[
			U_{m,n} = \Big[\cP( \kappa(\cdot/n^{1/d})) \setminus \{Z_{m+1,n},\ldots, Z_{N_n,n}\} \Big] 
\cup \{ V_{1,n},\ldots,V_{m-N_n,n} \},
\]
where, by convention, $\{Z_{m+1,n},\ldots, Z_{N_n,n}\}$ is empty if $N_n\le m,$ and $\{ V_{1,n},\ldots,V_{m-N_n,n} \}$ is empty if $N_n\ge m$.

By \eqref{E:StrongStabilizationRadius3}, for each $\ol\alpha\ge\ul\alpha>0$ and for each $\epsilon>0$, there is an $L>0$ such that
\[
			\sup_{\ul\alpha\le \alpha\le \ol\alpha} \sup_{n\in\N} \sup_{z\in B''_{n,L} } \p( \wt\rho_q^{\,\alpha r} ( \cP( \kappa(\cdot/n^{1/d})) , \{z\} ) > L) \le \epsilon.
\]
Note that here we use the fact that we only study one density function, namely, $\kappa$, so for values of $n$, $n\le n_0$, we choose $L>0$ individually and take the maximum in the end.

Also, for all $\epsilon>0$, for all $z\in B_n,$ and for all $K>0,$ there is an $n_0\in\N$ such that for all $n\ge n_0$, $\p( A'_n(K,z) ) \ge 1 - \epsilon$, where
\begin{align*}
				A'_n(K,z) = \{ [\{Z_{m+1,n},\ldots, Z_{N_n,n}\} \cup \{ V_{1,n},\ldots,V_{m-N_n,n} \} ] \cap B(z,K) = \emptyset \}.
\end{align*}
Indeed, this result follows from standard calculations as $\E{|m-N_n|} \le |m-n| + \E{|N_n-n|^2}^{1/2} \le h(n) + n^{1/2}$ and as the probability that a single point falls in $B(z,K)$ is bounded above by a constant times $n^{-1}$. 

In the last step, we combine these observations as follows. First, let $\epsilon>0$ be arbitrary but fixed. Then there is an $L^*>0$ such that 
\[
				\p(n^{1/d} X'\notin B''_{n,L^*}) \le \frac{\epsilon}{3}
				\]
and
\[
	\sup_{\ul\alpha\le \alpha\le \ol\alpha} \quad \sup_{n\in\N} \quad \sup_{z\in  B''_{n,L^*} } \p( \wt\rho_q^{\,\alpha r}( \cP( \kappa(\cdot/n^{1/d})) , \{z\} ) > L^*) \le \frac{\epsilon}{3}.
\]
Moreover, with $K^* = L^* + 2\mu(r)$, there is an $n_0\in\N$ such that for $z\in B_n$ and $n\ge n_0$, $\p( A'_n(K^*,z)^c ) \le \epsilon/3$.
Consequently, similar to \eqref{E:StrongStabilizationRadius1b}, if $n\ge n_0$ and if $L\ge L^*$, then
\begin{align*}
		&\p(\wt\rho_q^{\,\alpha r} (n^{1/d} \mX_m, n^{1/d} X') >L ) \\
		&\le \sup_{z\in B''_{n,L^*}}	\p( \wt\rho_q^{\,\alpha r}( U_{m,n}, \{z\} ) > L) + \frac{\epsilon}{3} \\
		&\le \sup_{z\in B''_{n,L^*}}	\p( \wt\rho_q^{\,\alpha r}( U_{m,n}, \{z\})>L, A'_n(K^*,z), \wt\rho_q^{\,\alpha r}( \cP( \kappa(\cdot/n^{1/d})) , \{z\} ) \le L^* ) + \epsilon = \epsilon,
\end{align*}
uniformly in $m\in J_n$ and $\ul\alpha\le \alpha \le \ol\alpha$. This shows (3) and completes the proof.
\end{proof}

Throughout the remainder of the appendix we consider a more general filtration as in \cite{hiraoka2018limit}. Examples for these filtrations are the {\v C}ech or the Vietoris-Rips filtration. The following principle will be important.
\begin{lemma}\label{L:IntersectionProperty}
Let $\cK$ and $\cK'$ be two simplicial complexes. Then $C_q(\cK)\cap C_q(\cK') = C_q(\cK\cap \cK')$. Moreover, $Z_q(\cK)\cap Z_q(\cK') = Z_q(\cK\cap\cK')$ and $B_q(\cK)\cap B_q(\cK') \supseteq B_q(\cK\cap\cK')$.
\end{lemma}
\begin{proof}
First, we consider the claim concerning the spaces $C_q$. The inclusion ``$\supseteq$'' is clear and we only prove ``$\subseteq $''. This inclusion can be deduced from the fact that $C_q$ is a free module over $\mathbb{F}_2$ generated by the corresponding $q$-simplices in the filtration. We can write $c\in C_q(\cK)\cap C_q(\cK')$ as $\sum_{i} a_i \sigma_i$, where $\sigma_i$ are $q$-simplices in $\cK$, $a_i\in\mathbb{F}_2$, and also as $\sum_{j} b_j \wt\sigma_j,$ where $\wt\sigma_j$ are $q$-simplices in $\cK'$, $b_j\in\mathbb{F}_2$. Hence, $\sum_{i} a_i \sigma_i - \sum_{j} b_j \wt\sigma_j =0$. If $\sigma_i\in \cK\setminus\cK'$, the coefficient $a_i$ is zero, as this basis element cannot occur in the filtration $\cK'$. The same holds in the other direction, if $\wt\sigma_j\in\cK'\setminus\cK$, $b_j$ is zero. 

The amendment $Z_q(\cK)\cap Z_q(\cK') = Z_q(\cK\cap\cK')$ follows immediately. Again the inclusion ``$\supseteq$'' is clear and we only prove ``$\subseteq$''. If $c\in Z_q(\cK)\cap Z_q(\cK')$, then by the above $c\in C_q(\cK\cap\cK')$ and by assumption $\partial c = 0$. Thus, $c\in Z_q(\cK\cap\cK') $ as desired. The inclusion concerning the boundary groups is immediate.
\end{proof}

We remark that $B_q(\cK)\cap B_q(\cK') \not\subseteq B_q(\cK\cap\cK')$ is possible. For instance, consider a situation where $\beta_0(\cK)=\beta_0(\cK')=1$ and where $\cK \cap \cK' = \{a,b\}$ for two 0-dimensional simplices $a,b$ such that $\cK\cap \cK'$ is a strict subset of $\cK$ and $\cK'$. (E.g., we can take two ``arc-like'' connected components represented by $\cK$ and $\cK'$ which only intersect in their endpoints.) Then we have $ B_0(\cK\cap\cK') = \{0\}$, but $B_0(\cK)\cap B_0(\cK')$ contains $a+b$ as a nontrivial element.

In the following assume that $P$ is a simple point cloud on $\R^d$ without accumulation points and $Q$ is a finite subset of $\R^d$ such that $P\cap Q = \emptyset$ and $Q\subseteq Q(z,L)$ for some $z\in\R^d$ and $L\in\R_+$.  Define
\begin{align*}
			\cK_{s,a} &= \cK_s( P\cap B(z,a)), \qquad \cK'_{s,a} = \cK_s((P\cup Q)\cap B(z,a)).
\end{align*}
Set $a^*= a^*(s) = \mu(s) + L$, where $\mu(s)$ is the upper bound on the diameter of a simplex in the filtration at time $s$, which is guaranteed by the assumptions of \cite{hiraoka2018limit} on the filtration.
Choose $a_1,a_2 \in \R$ with $a^*\le a_1 \le a_2$ such that  $C_0(\cK_{s,a_2}\setminus\cK_{s,a_1})$ contains exactly one additional basis element (a point) from $P$ and write
\[
				C_{q+1}( \cK'_{s,a_2} \setminus \cK'_{s,a_1}) = \langle \sigma_1,\ldots,\sigma_n \rangle.
\]
We can assume w.l.o.g.\ that the simplices are already in the right order, i.e.,
\begin{align}\label{E:ReprB_qDash}
				B_q (\cK'_{s,a_2}) = B_q( \cK'_{s,a_1}) \oplus \langle \partial\sigma_1,\ldots,\partial\sigma_i \rangle,
\end{align}
such that $\partial\sigma_j \neq 0$ mod $ B_q( \cK'_{s,a_1}) \oplus \langle \partial\sigma_1,\ldots,\partial\sigma_{j-1} \rangle$ for $j=1,\ldots,i$ and $\partial\sigma_j = 0$ mod $ B_q( \cK'_{s,a_1}) \oplus \langle \partial\sigma_1,\ldots,\partial\sigma_{i} \rangle$ for $j=i+1,\ldots, n$. As $a_1$ is sufficiently large, we have that each of the simplices $\sigma_j$ is also contained in $\cK_{s,a_2}$. Hence, as $B_q (\cK_{s,a})$ is a subspace of $B_q (\cK'_{s,a})$, we have that
\begin{align}\begin{split}\label{E:ReprB_q}
		&B_q (\cK_{s,a_2})= B_q( \cK_{s,a_1}) \oplus \langle \partial\sigma_1,\ldots,\partial\sigma_i \rangle \oplus \langle \partial\sigma_{j}: \text{ for } j\in J \rangle,
\end{split}\end{align}
where $J \subseteq \{i+1,\ldots,n\}$ can be empty. In particular,
\[
	\dim B_q (\cK'_{s,a_2}) - \dim B_q (\cK_{s,a_2}) = \dim B_q (\cK'_{s,a_1}) - \dim B_q (\cK_{s,a_1}) + (i - i - \# J).
\]
I.e., the map $a\mapsto \dim B_q( \cK'_{s,a}) / B_q (\cK_{s,a})$ is non-increasing if $a\ge a^*$.

Finiteness of the radius of weak stabilization $\rho^{(r,s)}(P,Q)$ follows directly from Lemma~\ref{L:StrongAndWeakStabilization} and Theorem~\ref{Thrm:StrongStabilization}. This argument relies on the finiteness of the radius of strong stabilization (Theorem~\ref{Thrm:StrongStabilization}), which, however, is not necessary for the finiteness of $\rho^{(r,s)}(P,Q)$. The following lemma gives a direct proof for the finiteness of $\rho^{(r,s)}(P,Q)$ using similar ideas as in the proof of Lemma 5.3 of Hiraoka et al.\ \cite{hiraoka2018limit}.

\begin{lemma}\label{L:MeaningfulRho}
The radius $\rho^{(r,s)}(P,Q)$ from \eqref{Def:WeakStabilizationRadius} is well-defined.
\end{lemma}
\begin{proof}
It is sufficient to consider $\rho^{(r,s)}_q(P,Q)$ for each $0\le q\le d-1$. One can use the geometric lemma to show that the nonnegative, integer-valued mappings
\begin{align}\label{E:MeaningfulRho1}
				a\mapsto \dim \frac{ Z_q (\cK'_{r,a})}{ Z_q (\cK_{r,a})} \text{ and } a\mapsto \dim \frac{ Z_q (\cK'_{r,a}) \cap B_q (\cK'_{s,a})}{ Z_q (\cK_{r,a})\cap B_q (\cK_{s,a})} 
\end{align}
are bounded above. Moreover, the map	$ Z_q (\cK'_{r,a_1}) / Z_q (\cK_{r,a_1})  \hookrightarrow  Z_q (\cK'_{r,a_2}) / Z_q (\cK_{r,a_2})$ is injective for all $0\le a_1\le a_2$, this follows from Lemma~\ref{L:IntersectionProperty}. Thus, there is an $a^*_1\in\R_+$ such that the first mapping in \eqref{E:MeaningfulRho1}, which is integer-valued, is constant for all $a\ge a^*_1$. 

Next, we show that the second mapping in \eqref{E:MeaningfulRho1} becomes also constant as $a\rightarrow \infty$. To this end, we first show that $a\mapsto \dim B_q (\cK'_{s,a}) /  B_q (\cK_{s,a})$ is constant for all $a\ge a^*_2$ for a certain $a^*_2\in\R_+$. This follows however from the non-increasing property (see paragraph right after \eqref{E:ReprB_q}) and the boundedness from below of this mapping. Now, we can return to the second mapping in \eqref{E:MeaningfulRho1} and one argues as in the proof of Lemma~\ref{L:StrongAndWeakStabilization} to show that the difference $\dim Z_q( \cK'_{r,a}) \cap B_q(\cK'_{s,a}) - \dim Z_q( \cK_{r,a}) \cap B_q(\cK_{s,a})$ is constant for all $a \ge a^*_1 \vee a^*_2$.

\end{proof}

\end{document}